\documentclass[10pt]{article}

\usepackage{amssymb,latexsym,amsmath, amscd,bbm}
\usepackage{amsfonts}
\usepackage[greek,english]{babel}
\usepackage[iso-8859-7]{inputenc}
\usepackage{amsthm}
\usepackage{bm}
\usepackage{graphicx}
\usepackage{mathrsfs} 
\usepackage{floatflt,setspace, wrapfig}
\usepackage{epstopdf}

\long\def\comment#1{}

\providecommand{\abs}[1]{\lvert \, #1 \, \rvert}

\newtheorem{theorem}{Theorem}

\newtheorem{lemma}{Lemma}

\newtheorem{proposition}{Proposition}

\newtheorem{corollary}{Corollary}

\def\qed{\unskip\nobreak\hfill\penalty50\hskip 3pt\hbox{}\nobreak
\hfill\hbox{\vrule width 4 pt height 10 pt}}

\newcommand{\const}{{\text{const.}}}

\newcommand{\op}{{\cal{L}}}
\newcommand{\f}{{\cal{F}}}
\newcommand{\B}{{\cal{B}}}

\newcommand{\R}{{\mathbb{R}}}

\newcommand{\no}{{\noindent}}
\newcommand{\Pro}{{\mathbb{P}}}

\addcontentsline{file}{sec_unit}{entry}

\begin{document}

\title{Return times distribution for Markov towers with decay of correlations}

\author{N T A Haydn\thanks{Mathematics Department, USC,
Los Angeles, 90089-1113. E-mail: $<$nhaydn@math.usc.edu$>$.
This work was partially supported by a grant from the NSF (DMS-0301910).}
and Y Psiloyenis\thanks{ E-mail: $<$psiloyen@usc.edu$>$.
This work was partially supported by a grant from the NSF (DMS-0301910).}}   

\date{\today} 

\maketitle

\begin{center}
{\Large Abstract}

\bigskip
\end{center}

In this paper we prove two results. First we show 
that dynamical systems with a $\phi$-mixing measure have in the limit 
Poisson distributed return times almost everywhere. We use the
Chen-Stein method to also obtain rates of convergence. Our theorem improves
on previous results by allowing for infinite partitions and dropping the requirement
that the invariant measure have finite entropy with respect to the given partition.
As has been shown elsewhere, the limiting distribution at periodic points
is not Poissonian (but compound Poissonian). Here we show that for all 
non-periodic points the return times are in the limit Poisson distributed. 
In the second part we prove that Lai-Sang Young's Markov Towers have 
Poisson distributed return times if the correlations decay for observables
that are H\"older continous and $\mathscr{L}^\infty$ bounded.

\section{Introduction}

Beginning with the Poincar\'e recurrence theorem, one of the main interests in studying deterministic dynamical systems has been to show
that the orbit of a typical point is on large timescales statistically regularly
distributed and orbit segments that are sufficiently separated are close to 
independently distributed.  In this paper we follow in this tradition and show
that for invariant measures that are $\phi$-mixing 
with respect to a possibly countably infinite partition the return times are
in the limit Poisson distributed. 

Interest in such questions go back to the 1940's when Doeblin~\cite{Doe} studied the 
Gauss map and its invariant measure. Later, in the 1970s Harris studied return times
for Markov processes and then around 1990 the interest of the return times
statistics became a central topic in dynamics. Using symbolic dynamics, Pitskel~\cite{Pit} 
proved for Axiom A maps the return times are in the limit Poisson distributed with
respect to equilibrium states for H\"older continuous potentials. Hirata~\cite{Hir1}
has a similar result using the Laplace transform which he then generalised later
in~\cite{Hir2}. Galves and Schmitt~\cite{GS} then came up with a 
technique to get results for the first entry or return time which they applied to 
$\psi$-mixing systems and where they also for the first time
provided error estimates. This method was then greatly extended by 
Abadi~\cite{A1,A2,A3} to $\phi$-mixing systems. Using a combinatorial argument 
improved error estimates were given in~\cite{AV1} for the first entry an return times
of $\phi$-mixing processes. For $\alpha$-mixing systems, the 
limiting entry and return times distribution was established in~\cite{AS}.
A combinatorial argument was used in~\cite{AV2,AV3}  to show
 that the limiting distribution is Poissonian for $\phi$-mixing measures if one takes the
  limit along a nested sequence of cylinders.
  In~\cite{LSV} multiple return times were shown to be Poisson distributed for 
  a class of intermittened systems.
 Recently Kifer has proven limiting results for 
simultaneous returns to cylinder sets, first~\cite{Kif11} an almost sure result using the Chen-Stein
method and then~\cite{Kif12} a complete classification with error terms.
Let us note that in~\cite{DGS} the Chen-Stein method was used to get the Poisson limiting
distribution for toral automorphisms where the limit is taken along sequences of ball-like sets.

Typically when entry times are Poisson distributed then so are the
return times. In fact, for arbitrary return or
entry times distribution there is a formula~\cite{HLV} that allows to translate the entry 
times distribution into the return times distribution and vice versa.

For attractors on manifolds (with $1$D unstable direction) which have a representation
by Young towers with exponentially decaying correlations,  Chazottes and Collet~\cite{CC}
 have shown that the entry times are Poisson distributed for the SRB measure. Here the 
 return sets are balls although the technique involves approximations by 
unions of cylinder sets.  Wasilewska~\cite{Was} extended this result to quite arbitrary measures
 on Young towers with polynomially decaying correlations. There, too, the return sets are 
 balls $B_\rho$ which are approximated by unions of cylinders. There the error terms
 decay with a negative power of $\abs{\log\rho}$. In particular for attractors this result
 applies  to SRB measures with polynomially decaying  correlations. See also~\cite{HW}.
 For an overview of distribution results of return times also see~\cite{H13}.

In this paper we consider maps that are $\phi$-mixing with respect to an invariant
measure and a partition which can be finite or countably infinite.
The purpose of the paper is threefold: (i) we devolop a more direct approach to
the method of Chen-Stein to obtain distribution results on return times, 
(ii)  the Poisson law we obtain is applicable to unions of cylinders rather than
single cylinder neighbourhoods, and (iii) we allow for infinite partitions and do not
require the entropy to be finite. Unlike the moment method which requires
the measure to have the stronger $\psi$-mixing property, the method of 
Chen-Stein requires us to only look at `two fold' mixing sets and this is what  makes it 
accessible to $\phi$-mixing measures. We also obtain rates of 
convergence. Since we show the limiting 
distribution for unions of cylinders whose total measures are required to decay at some rate,
this approach can be used to obtain limiting distribution results for metric balls in 
a metric space setting (Theorem~\ref{poisson.balls}).
Naturally we have to keep away from return sets that `look' periodic.
At periodic points the limiting distribution cannot be Poisson but is, as was
shown in~\cite{HV3}, compound Poisson distributed. In 
Corollary~\ref{returntimesdichotomy} we deduce that at all non-periodic points return
times are in the limit Poissonian.  

In Section~2 we set up the Chen-Stein method and then prove the main technical result
Proposition~\ref{mixingtheorem}. A similar method is used to prove 
Theorem~\ref{young.poisson}.
Most of the results of Sections~2 and~3 (in particular Theorem~\ref{maintheorem}
and Lemma~\ref{logsum}) also appeared in~\cite{Psi}.

In the second part (Section~4) of the paper we then look at Young towers and show that return and entry times
are in the limit Poisson distributed although we don't necessarily have the $\phi$-mixing
property for those systems. Since the invariant measure on a Young tower typically is 
not $\phi$-mixing (although it is $\alpha$-mixing), more delicate estimates are required
in order to obtain the limiting Poisson distribution along sequences of sets which are
unions of cylinders.

Let us note that it is crucial to select the return set to be some `regular' set like cylinders
 as Kupsa and Lacroix~\cite{KL,Lac}
have shown that any limiting distribution can be realised if one choses the return sets
appropriately. Also let us note that Kupsa has constructed an example of a symbolic
system over three elements which has positive entropy and whose first entry time
is not exponentially (with parameter one) distributed almost everywhere. 
This emphasises that despite
the plethora of existing results on the distribution of entry times, we cannot expect 
positive entropy systems to generically have Poisson distributed returns in the limit.

\section{Distribution for $\phi$-mixing systems}

Let $T$  be a map on $\Omega$ and $\mu$ a $T$-invariant probability measure on $\Omega$.
Let $\cal A$ be a finite or countably infinite measurable partition on $\Omega$. We put 
${\cal A}^n$ for its $n$th join $\bigvee_{j=0}^{n-1}T^{-j}{\cal A}$. We assume that the 
partition $\cal A$ is generating (i.e.\ the atoms of ${\cal A}^\infty$ consist of single points).

Throughout the paper we will assume that $\mu$ is  (right) {\em $\phi$-mixing}, that is there exists
a decreasing sequence $\phi(k)\rightarrow0$ (as $k\rightarrow\infty$) so that
$$
 \left|\frac{\mu(A\cap T^{-n-k}(B))}{\mu(B)}-\mu(A)\right|\le\phi(k)
$$
for all $A\in{\cal A}^n$, $B\in\sigma(\bigcup_{\ell\ge1}{\cal A}^\ell)$ ($\mu(B)>0$)
 and for all $n,k$ (see e.g.~\cite{D2}).
Let us note that there exists $\Lambda>0$ so that for any $n\in\mathbb{N}$ and $A\in{\cal A}^n$
one has $\mu(A)\le Ke^{-\Lambda n}$ for some constant $K$. For a proof of this fact see 
Abadi~\cite{A1} whose proof for
finite alphabets carries over to infinite alphabets without any change.

For a set $A\subset \Omega$ the {\em hitting time} $\tau_A: \Omega\to\mathbb N\cup\{\infty\}$
 is a random variable defined on the entire set $\Omega$ as follows
$$
\tau_A(x)=\inf \left\{k\ge1\colon T^k(x)\in A\right\}
$$
($\tau_A(x)=\infty$ if $T^kx\not\in A\;\forall k\in\mathbb{N}$).
If we narrow down the domain of $\tau_A$ to the set $A$ then $\tau_A$ is called the 
{\em return time} or first-return time.
According to Kac's theorem~\cite{K}  $\int_A\tau_A\,d\mu=1$ for any 
ergodic $T$-invariant probability measure $\mu$ and measurable $A\subset\Omega$ with
positive measure.
We then can define the induced map $\hat{T}_A:A\circlearrowleft$ given by 
$\hat{T}_A(x)=T^{\tau_A(x)}(x)\;\forall x\in A$,
and the $k^{th}$ return time $\tau_A^k$  by putting $\tau_A^1=\tau_A$
 ($k=1$) and forrecursively  $k>1$
$$
\tau_A^k(x)=\inf\left\{\ell>\tau_A^{k-1}(x)\colon T^\ell(x)\in A\right\}=\tau_A(\hat{T}_A^{k-1}(x))
$$
(for convenience we put $\tau_A^0=0$).
Following~\cite{AV3} the {\em period} of $A\subset \Omega$, under the map $T$, 
is defined to be
$$
r_A=\inf\{n \in\mathbb N | A\cap T^{-n}(A)\neq\emptyset\},
$$
or, equivalently, $r_A=\inf_{x\in A}\tau_A(x)$.
From the mixing property we conclude that $r_A\le\min\{\ell: \phi(\ell)<1\}$. 

For $A\in\sigma({\cal A}^n)$ (union of $n$-cylinders) let us define 
$$
\delta_A(j)=\min_{1\le w\le j\wedge n}\left\{\mu(A_{w}(A))+\phi(j-w)\right\},
$$ 
 where $A_{w}(A)\in\sigma({\cal A}^w)$ is smallest so that $A\subset A_{w}(A)$,
that is $A_w(A)=\bigcup_{B\in\mathcal{A}^w:\,B\cap A\not=\varnothing}B$.

\vspace{3mm}

\noindent {\bf Remark:} In a similar way one can define a measure $\mu$ to be 
{\em left $\phi$-mixing}\footnote{this is sometimes also called {\em reversed
$\phi$-mixing}.} if 
$$
 \left|\frac{\mu(A\cap T^{-n-k}(B))}{\mu(A)}-\mu(B)\right|\le\phi(k)
$$
for all $A\in{\cal A}^n$, $B\in\sigma(\bigcup_j{\cal A}^j)$ and $n,k$.
A right $\phi$-mixing measure is not necessarily also left $\phi$-mixing.
However the results in this paper on the distribution of return times (Theorems~1
and~2 and Corollary~1 and also Lemma~\ref{recurrenceestimates}) also
apply to left $\phi$-mixing systems since the techniques involved
are symmetric.
If the measure is left $\phi$-mixing then $\delta_A(j)$ has to be replaced by

$$
\hat\delta_A(j)=\min_{1\le w\le j\wedge n}\left\{\mu(A^{(w)}(A))+\phi(j-w)\right\}
$$
where $A^{(w)}(A)=T^{-(n-w)}T^{n-w}A\in\sigma(T^{-(n-w)}\mathcal{A}^w)$
is the smallest element in $\sigma(T^{-(n-w)}\mathcal{A}^w)$ which contains $A$
($w\le n$).

\begin{theorem}\label{maintheorem}
Let $\mu$ be a $T$-invariant  probability measure which is $\phi$-mixing with respect to 
a generating and at most countably infinite partition ${\cal A}$. 
Then there exists a constant $C_1$ so that
$$
 \left|\Pro\left(\tau_A^k>\frac{t}{\mu(A)}\right)-\sum_{i=0}^{k-1}e^{-t}\frac{t^i}{i!}\right|
\le C_1t (t\vee1)\inf_{\Delta>0}\left(\Delta\mu(A)+\sum_{j=r_A}^\Delta\delta_A(j)
+\frac{\phi(\Delta)}{\mu(A)}\right)|\log\mu(A)|.
$$
for all $k,n\in\mathbb{N}$ and $A\in\sigma({\cal A}^n)$.
\end{theorem}

\begin{theorem}\label{mainresult2}\cite{Psi}
Let  $\mu$ be a $\phi$-mixing $T$-invariant probability measure with respect to the
generating and at most countable infinite partition ${\cal A}$.
Let  $\eta\ge1$ be so that $n^\eta\phi(n)\rightarrow0$ as $n\rightarrow\infty$.
Let $K>0$. Then for  $A\in\sigma({\cal A}^n)$ a finite or infinite union of $n$-cylinders 
such that $|\log\mu(A)|\le Kn^\eta$ and $r_A>\frac{n}2$ the following applies:

\noindent {\textbf{(i) Exponential mixing rate:}}
Suppose $\phi(n)={\cal O}(\vartheta^{n})$, with $0<\vartheta<1$ and 
$\mu(A_w(A))=\mathcal{O}(\vartheta^w)$ for $w\le n$. Then there exists
 $\gamma=\gamma(\vartheta)>0$ 
and $C_2>0$ such that
\begin{equation}\label{exponential}
 \left|\Pro\left(\tau_A^k>\frac{t}{\mu(A)}\right)-\sum_{i=0}^{k-1}e^{-t}\frac{t^i}{i!}\right|
\le C_2t (t\vee1)e^{-\gamma n},\quad\forall t>0 \text{ and }\forall n\in\mathbb N.
 \end{equation}

\noindent {\textbf{ (ii) Polynomial mixing rate:}}
Suppose $\phi(n)={\cal O}(n^{-\beta})$ with $\beta>1+\eta$ and 
$\mu(A_w(A))=\mathcal{O}(w^{-\beta})$ for $w\le n$.
Then there exists $C_2>0$ such that 
\begin{equation}\label{polynomial}
 \left|\Pro\left(\tau_A^k>\frac{t}{\mu(A)}\right)-\sum_{i=0}^{k-1}e^{-t}\frac{t^i}{i!}\right|
\le C_2t (t\vee1)\frac{1}{n^{\beta-1-\eta}}, \quad\forall t>0 \text{ and } \forall n\in\mathbb N.
 \end{equation}
\end{theorem}

\no\textbf{Remarks:}
\vspace{2mm}

\noindent {\bf(I)} The statements of these two theorems also apply to left $\phi$-mixing measures.
In this case however the quantity $\delta_A(j)$ in Theorem~\ref{maintheorem}  has to be replaced by $\hat\delta_A(j)$ and in Theorem~\ref{mainresult2} the decay rate for $\mu(A_w(A))$ has to 
apply to $\mu(A^{(w)}(A))$ instead.
Here  we present the proof in the case when $\mu$ is right $\phi$-mixing.

\vspace{2mm}
 
\noindent {\bf (II)}
The assumption of Theorem \ref{mainresult2} that the period $r_A$ be greater than $\frac{n}{2}$ can be substituted with any other number of the order of $n$. This assumption is in place to ensure that the reference cylinder $A$ does not exhibit a periodic behavior. By its very definition, the set $A$ consists of points that travel together for at least $n$ iterates of the map $F$. In view of this property if the set $A$ revisited itself too early on by the means of a single point $x$ that would have caused an entire neighborhood of $A$ to fall into $A$ at that same iterate. Considering the extreme case, if the entire set falls into $A$ at the same iterate of $F$ that renders $A$ periodic. In this case the set $A$ would act like a ``trap''. By asking that more time passes by before any of $A$'s points comes back to $A$ we ensure that the system is nearer to the time where the set will start spreading all over the space, by virtue of the mixing properties that govern the dynamics.
In particular for cylinders around periodic points the limiting distribution of return times is a compound Poissonian distribution~\cite{HV3}.

\vspace{2mm}

\noindent {\bf (III)}
 Commenting on the assumption that  $|\log\mu(A_n)|\le Kn^\eta$  recall that in the finite entropy case,
 when $H({\cal A})<\infty$, the theorem of Shannon-MacMillan-Breiman~\cite{M} implies that for a.e.\ point 
 $x\in\Omega$ there exists $C>0$ such that 
\begin{equation}\label{shannonderivative}
|\log\mu(A_n(x))|\le Cn \quad \forall n\in\mathbb N,
\end{equation}
i.e.\ $\eta=1$, where we denote by $A_n(x)$ the $n$-cylinder centered at $x$.
  On the other hand, if $H({\cal A})<\infty$ and $\eta>1$ then we can give a rough estimate
on the set of cylinders that don't satisfy the condition $|\log \mu(A)|\le Kn^\eta$. 
Denote by $B(n)\subset {\cal A}^n$ the set of all the $n$-cylinders  $A$ that satisfy $|\log\mu(A)|>Kn^\eta$.
Then, since $H({\cal A}^n)=\sum_{A\in{\cal A}^n}\mu(A)|\log\mu(A)|\le nH({\cal A})$, we 
obtain
$$
nH({\cal A})\ge\sum_{A_n\in B(n)}\mu(A_n)|\log\mu(A_n)|
\ge\sum_{A_n\in B(n)}Kn^{\eta}\mu(A_n)\\
=Kn^{\eta}\mu(B(n))
$$
which implies
$$
\mu(B(n))\le \frac{H({\cal A})}{Kn^{\eta-1}}\le \frac{c}{n^{\eta-1}}.
$$
This shows that for $\eta>1$ as $n$ increases the exception set, or ``bad'' set, gets smaller. 
The bigger the $\eta$ we choose the bigger coverage we achieve, where the estimates hold,
but making $\eta$ larger that has a direct effect on the error estimates. 
As pointed out above, Abadi's result does not allow us to choose $\eta$ to be less than
$1$.

\vspace{4mm}

\noindent In the remainder of this section we will look at the return times
distribution for cylinder sets.
Let $x\in\Omega$ and denote by $\pi_n=r_{A_n(x)}$ the period of the $n$-cylinder
neighbourhood $A_n(x)\in\mathcal{A}^n$.  Since
$A_{n+1}(x)\cap T^{j}A_{n+1}(x)\subset A_{n}(x)\cap T^jA_{n}(x)\;\forall n,j$, 
one sees that   $\pi_n$ is an increasing sequence which implies that either
$\pi_n\to\infty$ or $\pi_n$ converges to a limit $\pi_\infty$ (which is a function of $x$).

In the finite case, $\pi_\infty<\infty$, the point $x$ is a periodic point with period $\pi_\infty$.
This follows from the fact that $x\in A_n(x)\cap T^{\pi_\infty}A_n(x)$ for all $n$ large enough. 
Since $\mathcal{A}$ is generating, the periodicity of $x$ follows from taking a limit 
$n\to\infty$ as $\{x\}=\bigcap_nA_n(x)$. 
For $\psi$-mixing measures it was shown in~\cite{HV3} that the limiting distribution of
$\Pro\left(\tau_A^k>\frac{t}{\mu(A)}\right)$ converges to 
the P\'olya-Aeppli compound Poisson distribution. For the limiting first return-time distribution
at a periodic point a complete description for $\phi$-mixing measures was given in~\cite{AV3}
where it was shown that the density has a point mass at $t=0$ of weight 
$\lim_{n\to\infty}\mathbb{P}_{A_n(x)}(\tau_{A_n(x)}=\pi_\infty)$ and is exponential
otherwise. This generalises a result of Pitskel~\cite{Pit} for equilibrium states on 
Axiom~A systems.

In the infinite case, when $\pi_n\to\infty$ as $n\to\infty$, 
$x$ is non-periodic and we can estimate $\delta_A$ as follows:
$$
\delta_{A_n(x)}(j)=\inf_{0\le k\le j\wedge n}\{\mu(A_k(x))+\phi(j-k)\}
\le Ke^{-\Lambda(j\wedge n)/2}+\phi(j/2)
$$
($k=j/2$), where we used the property that $\mu(A_k(x))\le Ke^{-\Lambda k}$ ($\Lambda>0$).
Hence, with some  $c_1$,
$$
\mathcal{E}_n(\Delta)=\sum_{j=\pi_n}^\Delta\delta_{A_n(x)}(j)
\le c_1e^{-\Lambda(\pi_\infty\wedge n)/2}+\sum_{j=\pi_\infty}^\infty\phi(j/2)\longrightarrow0
$$
as $n\to\infty$ if we assume that $\phi(j)$ is summable. 
Also note that if $\phi$ is summable then we get that $\lim_{j\to\infty}j\phi(j)=0$.
Hence there exist a sequence $\Delta_n$, $n=1,2,\dots$, so that $\phi(\Delta_n)/\mu(A_n(x))\to0$
and also $\Delta_n\mu(A_n(x))\to0$ as $n\to\infty$.

As a consequence of Theorem~\ref{maintheorem} we thus have the following result:

\begin{corollary}\label{returntimesdichotomy}
Let $\mu$ be a $\phi$-mixing w.r.t.\ the generating partition ${\cal A}$ that is at most 
countably infinite.
Assume $\phi(j)$ is summable. If $x\in\Omega$ is not periodic, then
$$
\Pro\left(\tau_{A_n(x)}^k>\frac{t}{\mu(A_n(x))}\right)
\longrightarrow\sum_{i=0}^{k-1}e^{-t}\frac{t^i}{i!}
$$
as $n\to\infty$ for all $t>0$.
\end{corollary}

\noindent This is sometimes expressed using the counting function 
$\zeta_A^t=\sum_{j=0}^m\chi_A\circ T^j$, where $m=[t/\mu(A)]$ and $\chi_A$ is 
the characteristic function of $A$. Then
$\mathbb{P}(\tau_A^k>t/\mu(A))=\sum_{i=0}^{k-1}\mathbb{P}(\zeta_A^t=i)$ and 
 the statement of the corollary reads
$$
\Pro\left(\zeta_{A_n(x)}^t=k\right)
\longrightarrow e^{-t}\frac{t^k}{k!}
$$
as $n\to\infty$ for all non-periodic $x\in\Omega$ and all $t>0$.
As remarked earlier, this result equally applies to left $\phi$-mixing measures.


\subsection{Application}

As an application of Theorem~\ref{mainresult2} we will indicate how one can obtain
the limiting distribution for metric balls for maps on metric spaces. 
We will still require that there be a generating partition with respect to which the measure
is $\phi$-mixing. The balls will then be approximated by unions
of cylinders. This approach was also used by Pitskel~\cite{Pit} for toral automorphisms
on $\mathbb{T}^2$ and  in~\cite{H00} for rational maps.

Let $T$ be a map on a metric space $\Omega$ and let $\mathcal{A}=\{A_j:j\}$ a generating finite
or countable infinite partition of $\Omega$, that is $\Omega=\bigcup_jA_j$ and
$A_j\cap A_i=\varnothing$ for $i\not= j$. 
As before we denote by $\mathcal{A}^n$ the $n$th joint of the partition.
Assume there is a $T$-invariant probability measure $\mu$ on $\Omega$. Then we put for
parameters $t>0$ and radii $\rho>0$
$$
\zeta_{B_\rho(x)}^t=\sum_{j=0}^{m}\chi_{B_\rho(x)}\circ T^j
$$
for the counting function of the returns to the metric ball $B_\rho(x)$ in the space $\Omega$,
where $m= [t/\mu(B_\rho(x)]$.

\begin{theorem}\label{poisson.balls}
Let $\mu$ be an invariant measure on the metric space $\Omega$ and suppose there
is a partition (finite or countably infinite) $\mathcal{A}$. Let $x\in\Omega$ and assume the
following conditions are satisfied:  \\
(i) $\mu$ is $\phi$-mixing with rate $\phi(k)$ decaying at least polynomially with power larger than $2$;\\
(ii) $\mbox{diam}(\mathcal{A}^n)$ decays exponentially fast  as $n\to \infty$;\\
(iii) There exists $w>1$ such that 
$\frac{\mu(B_{\rho+\rho^w}(x))}{\mu(B_\rho(x))}\longrightarrow1$ as $\rho\to0^+$ almost everywhere;\\
(iv) $\mu$ has finite and positive dimension almost everywhere;\\
(v) $r_{B_\rho(x)}\ge \mbox{const.}|\log \rho|$ for small enough $\rho$.

Then
$$
\mathbb{P}\left(\zeta_{B_\rho(x)}^t=k\right)\longrightarrow e^{-t}\frac{t^k}{k!}
$$
as $\rho\to0^+$ for almost every $x\in\Omega$ and $k\in\mathbb{N}_0$.
\end{theorem}

\noindent {\bf Proof.} We approximate the balls $B_\rho(x)$ by unions of cylinders. 
By assumption~(ii) there exists a $v\in(0,1)$ such that $\mbox{diam}(\mathcal{A})\le v^n$
(for $n$ large enough). 
Let $n=\left[w\frac{\log \rho}{\log v}\right]+1$, fix $x$ and denote by
$$
C_{\rho,n}^t=\bigcup_{A\in \mathcal{A}^n: A\cap B_\rho(x)\not=\varnothing}A
$$
the smallest union of $n$-cylinders that contains $B_\rho(x)$. By assumption~(iv)
we have $|\log \mu(B_\rho(x))|\le c_1|\log \rho|$ for some constant $c_1<\infty$
and consequently the sets $C_{\rho,n}^t\in\sigma(\mathcal{A}^n)$ satisfy the assumption of 
Theorem~\ref{mainresult2} for $\eta=1$. By assumption~(v) we have
$r_{B_\rho(x)}\ge\mbox{const.}n$ thus satisfying the short return times condition.
Hence we obtain by Theorem~\ref{mainresult2}
that $
\mathbb{P}\left(\zeta_{C_{\rho,n}}^t=k\right)\longrightarrow e^{-t}\frac{t^k}{k!}
$
as $\rho\to0$ (and $n\to\infty$).

By assumption~(iii) on the regularity of the measure $\mu$ we have 
$$
\left|\mathbb{P}\left(\zeta_{B_{\rho+\rho^w}}=k\right)
-\mathbb{P}\left(\zeta_{B_{\rho}}=k\right)\right|\le\left[\frac{t}{\mu(B_\rho(x))}\right]
\mu(B_{\rho+\rho^w}\setminus B_{\rho})\longrightarrow0
$$
as $\rho\to0$.
Since $B_\rho(x)\subset C_{\rho,n}^t\subset B_{\rho+v^n}(x)\subset B_{\rho+\rho^w}$
(as $v^n<\rho^w$)
we obtain $ \mathbb{P}\left(\zeta_{B_{\rho}(x)}^t=k\right)\longrightarrow e^{-t}\frac{t^k}{k!}$
\qed

\vspace{2mm}

\noindent {\bf Remarks:}

\vspace{2mm}

\noindent {\bf (I)} The requirement~(i) that $\mu$ is $\phi$-mixing appears somewhat
artificial, but it can occur in the following simple way: An Anosov map $T$ on a manifold $\Omega$
admits the construction of an arbitrarily fine Markov partition $\mathcal{A}$ which then can be used to 
model the dynamics of $T$ by the shift transform $\sigma$ a subshift of finite type $\Sigma$. 
The projection $\pi:\Sigma\to\Omega$ semiconjugates the shift transform 
$\sigma:\Sigma\circlearrowleft$ to the map $T:\Omega\circlearrowleft$; that is 
$\pi\circ\sigma=T\circ\pi$. A $\phi$-mixing measure
$\nu$ on $\Sigma$ then maps  to a $\phi$-mixing measure $\mu=\pi^*\nu$  on
$\Omega$. Theorem~3 then implies  that the limiting return times
distribution for metric balls is Poissonian (provided conditions~(iii)--(v) are met).

\vspace{2mm}

\noindent {\bf (II)} If $\Omega$ is a manifold and $\mu$ is an absolutely continuous measure then
 the regularity condition~(iii)
  $\frac{\mu(B_{\rho+\rho^w}(x))}{\mu(B_\rho(x))}\longrightarrow1$ as $\rho\to0^+$ 
 is satisfied everywhere for any $w>1$.
 
 \vspace{2mm}

\noindent  {\bf (III)} Condition~(v) on the short returns is satisfied for many measures. For instance
in~\cite{CC}, Lemma~4.1, it was shown that for the SRB measure on codimension one attractors
 with exponentially decaying tails there exists an $\mathfrak{a}>0$ so that the 
 measure of the set of very short returns
 $$
 \mathcal{V}_\rho =\{x\in \Omega: r_{B_\rho(x)}>\mathfrak{a}\abs{\log \rho}\}
 $$ 
 is bounded by  $\mu(\mathcal{V}_\rho)=\mathcal{O}(\rho^a)$ for some $a>0$. Although the proof
 uses Young towers it does not rely on the decay of corellations or a mixing property. This was 
 in~\cite{Was, HW} extended to invariant measures for more general maps that allow for 
 a Young tower construction with polynomially decaying tails where one gets
 the estimate $\mu(\mathcal{V}_\rho)=\mathcal{O}(\abs{\log\rho}^{-a})$ for some $a>0$.
 In both cases every point  $x\not\in\mathcal{V}_\rho$ satisfies condition~(v).
 
  \vspace{2mm}

\noindent  {\bf (IV)} The theorem cannot in general directly be applied to systems that 
are modelled by a Young tower since the invariant measure is only $\alpha$-mixing
 and not necessarily $\phi$-mixing (see equation~(\ref{young.mixing2})). 
 A more elaborate method will be used to exploit the $\mathscr{L}^1$ convergence 
 of the densities (see Theorem~\ref{young.poisson}).

\section{Proof of Theorem~\ref{maintheorem}}

\subsection{Short returns}

Abadi has shown that for $\phi$-mixing systems the measure of cylinder sets
decay exponentially, i.e.\ there are strictly positive constants $K$ and $\Lambda$ 
such that $\mu(A)\le Ke^{-\Lambda n}$
 for any integer $n\in\mathbb N$ and any $n$-cylinder $A$.
 Recall that $\delta_A(k)=\min_{1\le w<k}\left\{\mu(A_{w}(A))+\phi(k-w)\right\}$
 where $A_{w}(A)\in\sigma({\cal A}^w)$ is smallest so that $A\subset A_{w}(A)$.

Recall that the period $r_A$ of the set $A$ is defined as the smallest $j$ for which 
$A\cap T^{-j}(A)\neq\emptyset\}$.

\begin{lemma}\label{recurrenceestimates}
$\mathbb{P}_A(\tau_A\le t)\le \sum_{j=r_A}^{t}\delta_A(j)$.
\end{lemma}

\begin{proof} For numbers $w_j\le j$ we have
\begin{eqnarray*}
\mu\left(A\cap\left\{\tau_A\le t\right\}\right)&
=&\sum_{j=r_A}^{t}\mu\left(A\cap\left\{\tau_A=j\right\}\right)\\
&\le&\sum_{j=r_A}^{t}\mu\left(A_{w_j}(A)\cap T^{-(n-j)}A\right)\\
&\le&\sum_{j=r_A}^{t}\mu(A)\delta_A(j)
\end{eqnarray*}
using the right $\phi$-mixing property and optimising for $w_j$.
The result how follows.
\end{proof}

\noindent In the same way one proves that
$\mathbb{P}_A(\tau_A\le t)\le \sum_{j=r_A}^{t}\hat\delta_A(j)$
if $\mu$ is left $\phi$-mixing since then 
$\mu\left(A\cap\left\{\tau_A=j\right\}\right)
\le\mu\left(A\cap T^{-(n-j)}A^{(w)}(A)\right)\le\hat\delta_A(j)$ for the optimal choice
of $w\in[1,n]$.

\subsection{The Stein method}

Stein's method of proving limiting theorems was first introduced by Stein~\cite{Ste}
for the Central Limit Theorem and then subsequently developed for the Poisson
 distribution~\cite{BC1,BC2}. As mentioned before this method has been used
in dynamics several times: Abadi~\cite{A3} used it by way of a result in~\cite{AGG}
to obtain the Poisson distribution for cylinder sets in $\phi$-mixing systems.
Denker, Gordin and Sharova~\cite{DGS} used the Chen-Stein method to 
obtain the Poisson distribution for limiting return times to ball-like sets 
for torus maps. Their approach involved extensive use of harmonic analysis.
Here we develop a more practical  approach that does not use~\cite{AGG} and
 does not require the target set to be a single cylinder, but could possibly be an infinite
 union of cylinders. Also, since entropy does not play any role, this approach works
 for infinite entropy systems and infinite alphabets. In the following we give a 
short description of the method as it is relevant for our purpose.

Let $\mu$ be a probability measure on $\mathbb{N}_0$ which is equipped with the 
power $\sigma$-algebra $\B_{\mathbb{N}_0}$. Additionally we denote by
$\mu_0$ the Poisson-distribution measure with mean $t$, i.e.\
$\Pro_{\mu_0}(\{k\})=\frac{e^{-t}t^k}{k!}$  $\forall k\in \mathbb{N}_0$.
Also let $\f$  be the set of all real-valued functions on  $\mathbb{N}_0$. The Stein 
operator ${\cal S}:\f\rightarrow\f$ is defined by
\begin{equation}\label{steinoperator}
{\cal S}f(k)=tf(k+1)-kf(k),\quad\text{ }  \forall k\in \mathbb{N}_0.
\end{equation}
The Stein equation 
\begin{equation}\label{steineq}
{\cal S}f=h-\int_{\mathbb{N}_0}h\,d\mu_0
\end{equation}
for the Stein operator in (\ref{steinoperator}), has a solution $f$ for each 
$\mu_0$-integrable $h\in\f$ (see~\cite{BC1}). 
The solution $f$ is unique except for $f(0)$, which can be chosen arbitrarily. 
Moreover $f$ can be computed recursively from the Stein equation, namely~\cite{BC1}:
\begin{eqnarray}
f(k)&=&\frac{(k-1)!}{t^k}\sum_{i=0}^{k-1} \left( h(i)-\mu_0(h)\right)\frac{t^i}{i!}\label{steinrepresentation1}\\
&=&-\frac{(k-1)!}{t^k}\sum_{i=k}^{\infty} \left( h(i)-\mu_0(h)\right)\frac{t^i}{i!} , 
\quad\text{} \forall k\in\mathbb{N} \label{steinrepresentation}
\end{eqnarray}
In particular,  if $h:\mathbb{N}_0\rightarrow\mathbb{R}$ is bounded then so is the associated Stein solution $f$.

\begin{proposition}~\cite{BC1}
A probability measure $\mu$ on $(\mathbb{N}_0,\B_{\mathbb{N}_0})$ is Poisson 
(with parameter $t$)
 if and only if 
$$
\int_{\mathbb{N}_0}{\cal S}f\,d\mu=0 \quad\text{for all bounded functions  } f:\mathbb{N}_0\rightarrow\R.
$$
\end{proposition}

A probability measure $\mu$ on $(\mathbb{N}_0,\B_{\mathbb{N}_0})$ which approximates 
the  Poisson distribution $\mu_0$ can be estimated as follows:
\begin{equation}\label{finalformstein}
|\mu(E)-\mu_0(E)|=\left | \int_{\mathbb{N}_0}{\cal S}f\,d\mu\right |
=\left | \int_{\mathbb{N}_0}\left (t f(k+1)-kf(k)\right)d\mu\right |
\end{equation}
where $E\subset\mathbb{N}_0$ and  $f$ is the Stein solution that corresponds to the 
indicator function $\chi_E$.  Sharp bounds for the quantity on the right-hand side
 of (\ref{finalformstein}) is what one is after when the Stein method is used for Poisson approximation.

\begin{lemma}\label{logsum}
For the Poisson distribution $\mu_0$, the Stein solution of the Stein equation ($\ref{steineq}$) that corresponds to the indicator function $h=\chi_E$, with $E\subset\mathbb{N}_0$, satisfies
\begin{equation}\label{estimatessteinsolution}
\left|f_{\chi_E}(k)\right|\le
\begin{cases}
    1  \quad&\text{ if } k \le t\\
     \frac{2+t}{k} \quad&\text{ if } k >t\;.
\end{cases}
\end{equation} 
In particular
\begin{eqnarray}
\sum\limits_{i=1}^{m}\left|f_{\chi_E}(i)\right|
&\le& \begin{cases} m \quad&\text{if } m\le t \\
t+(2+t)\log\frac{m}{t} \quad&\text{if } m>t\;.
\end{cases}
\end{eqnarray}

\end{lemma}

\begin{proof}

We consider the two cases (i)~$k>t$ and (ii)~$k\le t$.

(i) $k>t$:
For $h=\chi_{E}$, from the representation ($\ref{steinrepresentation}$) for the Stein solution we have 
$$
 f_{\chi_E}(k) =-\frac{(k-1)!}{t^k}\sum_{i=k}^{\infty} \left( h(i)-\mu_0(h)\right)\frac{t^i}{i!}.
$$
Therefore,
\begin{eqnarray}\label{boundedsteinsol}
\left| f_{\chi_E}(k)\right| &\le&\frac{(k-1)!}{t^k}\sum_{i=k}^{\infty} \left| h(i)-\mu_0(h)\right|\frac{t^i}{i!}\notag\\
&\le& \frac{(k-1)!}{t^k}\sum\limits_{i=k}^{\infty} \frac{t^i}{i!}\notag\\
&=& \frac{(k-1)!}{t^k}\frac{t^k}{k!}\left(1+\sum\limits_{i=1}^{\infty}\frac{t}{k+1}\frac{t}{k+2}\dots\frac{t}{k+i}\right).
\end{eqnarray}
If $i>t$ then each term in the infinite sum in $(\ref{boundedsteinsol})$ is no greater than 
$(\frac{1}{2})^{i-t}$. If $i\le t$, all terms in the sum in  $(\ref{boundedsteinsol})$ are clearly no greater than 1. Hence
$$
\left| f_{\chi_E}(k)\right| 
\le\frac{(k-1)!}{t^k}\frac{t^k}{k!}\left(1+t +\sum\limits_{i=1}^{\infty} \left(\frac{1}{2}\right)^i\right)
=\frac{2+t}{k}.
$$

(ii) $k\le t$:
Using the alternative representation ($\ref{steinrepresentation1}$) for the Stein solution $f_{\chi_E}$, 
this time, we get
$$
|f_{\chi_E}(k)| \le\frac{(k-1)!}{t^k}\sum_{i=0}^{k-1} \left| h(i)-\mu_0(h)\right|\frac{t^i}{i!}
\le\frac{(k-1)!}{t^k}\sum_{i=0}^{k-1} \frac{t^i}{i!} 
\le\frac{(k-1)!}{t^k}\frac{t^{k-1}}{(k-1)!}k
\le1
$$
as the sequence $\{\frac{t^j}{j!}\}_{j\in\mathbb N}$ is increasing for $j\le t$ and decreasing for $j>t$. This completes the proof of  inequality~$(\ref{estimatessteinsolution})$.
The second statement is now obvious for $m\le t$. On the other hand if 
$m>t$ then it follows from the inequality
$\sum_{i=t+1}^{m}\frac{1}{i}\le \log\frac{m}{t}$.
\end{proof}

\subsection{Return times distribution}

Now we want to approximate  the function 
$\Pro(\tau_A^k\le m)$ for all $k\ge 1$ and all $m\in \mathbb{R}^+$.
Let $A\in\sigma({\cal A}^n)$ and denote by $W_m(x)$ the number of visits of the orbit 
$\left\{T(x),T^2(x),\dots ,T^{[m]}(x)\right\}$ to the set $A$, i.e. 
$$
W_{[m]}(x)=\sum\limits_{j=1}^{[m]}\chi_A(T^j(x))
$$
where $\chi_A$ is the characteristic function of the set $A$ that is
$\chi_A(x)=1$ if $x\in A$ and $\chi_A(x)=0$ otherwise (and $[m]$ is the integer part of $m$).
 Then
$$
\Pro(\tau_A^k\le [m])=1-\Pro(\tau_A^k>[m])=1-\Pro(W_{[m]}<k)               
$$

Therefore, our problem of approximating the distribution of $\tau^k_A$ becomes equivalent to approximating the distribution of $W_{m}$ for all $m\in\mathbb{N}$. 
The Poisson parameter $t$ is the expected value of $W_m$ (i.e.\ $t=\mu(W_m)$).
If we put $p_i= \mu(T^{-i}A)=\mu(A)\quad\forall i=1,2,\dots$, then 
$$
t=\mu(W_m)=\sum\limits_{i=1}^m\mu\left(\chi_AT^i\right)=\sum_{i=1}^{m} p_i=m\mu(A)
$$
i.e.\ $m=[t/\mu(A)]$.
If $h=\chi_E$ with $E$ an arbitrary subset of the positive integers, $E\subset\mathbb{N}_0$,
then we obtain from~(\ref{finalformstein})
$$
\left |\int_{\mathbb{N}_0}{\cal S}f\,d\mu\right |
=\left | \int_{\mathbb{N}_0} hd\mu-\int_{\mathbb{N}_0} hd\mu_0\right | =\left|\Pro(W_m\in E)-\mu_0(E)\right|
$$
and in turn, since the Stein Operator ${\cal S}$ for the Poisson distribution is given by
$(\ref{steinoperator})$,  we obtain
$$
\left|\Pro(W_m\in E)-\mu_0(E)\right|=\left|E\left(t f(W_m+1)-W_mf(W_m)\right)\right|
\quad\forall E\subset \mathbb{N}_0.
$$
Notice that the difference $\left|\Pro(W_m\in E)-\mu_0(E)\right|$ above gives exactly the error of the Poisson approximation. We hence estimate
\begin{align}\label{errornewrepresentation}
\left|\Pro(W_m\in E)-\mu_0(E)\right|
 &=\left|t\mathbb{E}f(W_m+1)-\mathbb{E}\left(\sum_{i=1}^{m}I_if(W_m)\right) \right|\notag\\  
&=\left|\sum_{i=1}^{m}p_i\mathbb{E}f(W_m+1)-\sum_{i=1}^{m}p_i\mathbb{E}(f(W_m)|I_i=1)\right|\notag\\
&=\left|\sum_{i=1}^{m}p_i\left(\mathbb{E}f\left(W_m+1\right)-\mathbb{E}\left(f(W_m)|I_i=1\right)\right)\right|\notag\\
&=\sum_{i=1}^{m}p_i\left(\sum_{a=0}^{m}f(a+1)\Pro(W_m=a)-\sum_{a=0}^{m}f(a)\Pro(W_m=a|I_i=1)\right)\notag\\
&= \sum_{i=1}^{m}p_i\sum_{a=0}^{m}f(a+1)\epsilon_{a,i},
\end{align} 
where we put $I_i(x)=\chi_AT^i(x)$ the characteristic function of the set $T^{-i}A$ and
\begin{equation}\label{errorterm1}
\epsilon_{a,i}=\left|\Pro(W_m=a)-\Pro(W_m=a+1|I_i=1)\right|.
\end{equation}
The function $f$ above is the solution of the Stein equation~$(\ref{steineq})$ that 
corresponds to the indicator function $h=\chi_E$ in the Stein method. In fact bounds on
$f$ have been obtained in Corollary~$\ref{logsum}$.

Now, in view of the new representation for $\left|\Pro(W_m\in E)-\mu_0(E)\right|$
 we need to look at the term $\epsilon_{a,i}$ more closely. 
 If we put $W_m^i=W_m-\chi_A\circ T^i$ then
 the mixing condition yields the following estimates on $\epsilon_{a,i}$:
\begin{eqnarray*}
\epsilon_{a,i}&=&\left|\Pro(W_m=a)-\Pro(W_m=a+1|I_i=1)\right| \\
                    &=&\left|\Pro(W_m=a)-\frac{\Pro\left(\{W_m^i=a\}\cap T^{-i}A\right)}{\mu(A)}\right|\\   
                   &=&\left|\Pro(W_m=a)-\frac{\Pro(W_m^i=a) \mu(A)+\epsilon'_{a,i}}{\mu(A)}\right|\\
&\le&\left|\Pro(W_m=a)-\Pro(W_m^i=a)\right| +\frac{\xi_a}{\mu(A)}   \label{xi_alpha},                  
\end{eqnarray*}
where $\epsilon'_{a,i}=\Pro\left(\{W_m^i=a\}\cap T^{-i}A\right)-\Pro(W_m^i=a) \mu(A)$
($\epsilon'_{a,i}=0$ if all $I_j$ are independent) and
 $\xi_a=\max_i\left|\Pro(\{W_m^i=a\}\cap T^{-i}A)-\Pro(W_m^i=a) \mu(A)\right|$. 
The bound on $\epsilon_{a,i}$ has two terms, the first of which is
$$
\left|\Pro(W_m=a)-\Pro(W_m^i=a)\right|\le \Pro(I_i=1)=\mu(A).
$$
The second term, which contains $\xi_a$, is the error due to dependence for which we get 
estimates in Proposition~$\ref{mixingtheorem}$ below.

\begin{proposition} \label{mixingtheorem}
There exists a positive constant $C$ so that for all $n\in\mathbb{N}$ and
for all $A\in\sigma({\cal A}^n)$  the following estimate holds true 
$$
\bigg|\Pro\left(\{W_m^i=a\}\cap T^{-i}A\right)-\Pro(W_m^i=a)\mu(A)\bigg|
\le C\mu(A)\inf_{\Delta>0}\left( \Delta \mu(A)+\sum_{j=r_A}^\Delta\delta_A(j)
+\frac{\phi(\Delta)}{\mu(A)}\right)
$$
where $W_m=\sum_{j=1}^m\chi_A\circ T^j$ and $W_m^i=\sum_{\substack{1\le j\le m \\ j\neq i}}\chi_A\circ T^j$.
\end{proposition}

\begin{proof}
Let $\Delta<\!\!<m$ be a positive integer (the halfwith of the gap) and put for every $i\in(0,m]$
\begin{align}W_m^{i, -}&=\sum\limits_{j=1}^{i-(\Delta+1)}\chi_A\circ T^j, &
W_m^{i, +}&=\sum\limits_{j=i+\Delta+1}^{m}\chi_A\circ T^j, \nonumber\\
 U_m^{i, -}&=\sum\limits_{j=i-\Delta}^{i-1}\chi_A\circ T^j, &
U_m^{i,+}&=\sum\limits_{j=i+1}^{i+\Delta}\chi_A\circ T^j, \nonumber\\
 U_m^{i}&= U_m^{i, -} + U_m^{i,+}, &
 \tilde{W_m^{i}}&= W_m^i-U_m^{i} =W_m^{i, -} + W_m^{i, +} \nonumber
 \end{align}
with the obvious modifications if $i<\Delta$ or $i>m-\Delta$.
With these  partial sums we distinguish between the hits that occur near  the $i^{th}$ iteration, namely $U_m^{i,-}$ and  $U_m^{i,+}$, and the hits that occur away from the $i^{th}$ iteration, namely
$ W_m^{i, -}$ and $ W_m^{i, +}$.

The `gap' of length $2\Delta+1$ allows us to use the mixing property in the terms $W_m^{i,\pm}$
 and its size will be determined later by optimising the error term. 

We then have, for $0\le a \le m-1$, $a\in\mathbb N_0$, that
\begin{eqnarray*}
\Pro(\{W_m=a+1\}\cap T^{-i}A)&=&\Pro(\{W_m^i=a\}\cap T^{-i}A) \\
&=&\sum_{\substack{\vec{a}=(a^-,a^{0,-},a^{0,+},a^+)\\ \text{s.t } |\vec{a}|=a}}
\Pro\big(\{W_m^{i, \pm}=a^\pm\}\cap \{U_m^{i,\pm}=a^{0,\pm}\}\cap T^{-i}A \big)                                      
\end{eqnarray*}
 (intersection of five terms). For $0\le a\le m-1$ we have
$$
\bigg| \Pro\left(\{W_m^i=a\}\cap T^{-i}A\right)-\Pro\left(W_m^i=a\right)\mu(A)\bigg|\le R_1+R_2+R_3
$$
and will estimate the three terms
\begin{eqnarray*}
R_1&=&\left|\Pro\left(\{W_m^i=a\}\cap T^{-i}A\right)-\Pro\left(\{\tilde{W_m^i}=a\}\cap T^{-i}A\right)\right|\\
R_2&=&\left|\Pro\left(\{\tilde{W_m^i}=a\}\cap T^{-i}A\right)-\Pro\left(\tilde{W_m^i}=a\right)\Pro\left(I_i=1\right)\right|\\
R_3&=&\left|\Pro\left(\tilde{W_m^i}=a\right)-\Pro\left(W_m^i=a\right)\right|\mu(A)
\end{eqnarray*}
separately as follows.

\vspace{3mm}

\noindent\textbf{Estimate of $R_1$: } Here we show that short returns are rare when 
conditioned on $T^{-i}A$. Observe that
\begin{eqnarray*}
\{W_m^i=a\}\cap T^{-i}A
&\subset& \left(\{\tilde{W_m^i}=a\}\cap T^{-i}A\right)\cup\left(\{U_m^i>0\}\cap T^{-i}A\right)\\
\{\tilde{W_m^i}=a\}\cap T^{-i}A&\subset& \left(\{W_m^i=a\}\cap T^{-i}A\right)
\cup\left(\{U_m^i>0\}\cap T^{-i}A\right).
\end{eqnarray*}
Since $U_m^i>0$ implies that either $U_m^{i,+}>0$ or $U_m^{i,-}>0$ we get
$$
\big|\Pro\big(\{W_m^i=a\}\cap T^{-i}A\big)-\Pro\big(\{\tilde{W_m^i}=a\}\cap T^{-i}A\big)\big|
\le \Pro\big( \{U_m^i>0\}\cap T^{-i}A\big)\le b^-_i+b^+_i
$$
where
$$
b^-_i=\Pro\big( \{U_m^{i,-}>0\}\cap T^{-i}A\big)\quad\text{and}\quad
b^+_i=\Pro\big( \{U_m^{i,+}>0\}\cap T^{-i}A\big).
$$
We now estimate the two terms,  $b^-_i$ and $b^+_i$, separately as follows:

\noindent {\bf (i)} Estimate of $b^+_i$:
By  Lemma~\ref{recurrenceestimates}
  \begin{eqnarray*}
 b_i^+&=&\Pro\big( \{U_m^{i,+}>0\}\cap T^{-i}A\big)\\
 &= & \Pro(U_m^{i,+}>0|I_{i}=1)\mu(A)\\
  &=&\Pro_A(\tau_A\le\Delta)\mu(A)\\
 &\le& C\mu(A)\sum_{j=r_A}^\Delta\delta_A(j).
  \end{eqnarray*}

\noindent {\bf (ii)} Estimate of $b^-_i$:
 If $U_m^{i,-}>0$ then $\{{U_m^{i,-}>0}\}\subset \bigcup_{k=1}^{\Delta}T^{-(i-k)}A$ and therefore
$$
 \Pro\left( \{U_m^{i,-}>0\}\cap T^{-i}A\right)
 \le\mu\left(T^{-i}A\cap\bigcup_{k=1}^{\Delta}T^{-(i-k)}A\right)
$$
 We show the following symmetry
$$
\mu\left(T^{-i}A\cap\bigcup_{k=1}^\Delta T^{-(i-k)}A\right)
=\mu\left(T^{-i}A\cap\bigcup_{k=1}^\Delta T^{-(i+k)}A\right)
$$
For that purpose let $S_i=\bigcup_{k=1}^\Delta J_{i,k}$ where $J_{i,k}=T^{-i}A\cap T^{-(i-k)}A$
and similarly $\tilde{S}_i=\bigcup_{k=1}^\Delta \tilde{J}_{i,k}$,  $\tilde{J}_{i,k}=T^{-i}A\cap T^{-(i+k)}A$.
We now want to show that $\mu(S_i)=\mu(\tilde{S}_i)$.
We decompose $S_i$ into a disjoint union as follows:
$$
S_i= \bigcup_{k=1}^\Delta V_{i,k},
$$
where
$$
V_{i,k}= J_{i,k}\setminus\bigcup_{j=1}^{k-1}J_{i,k}\cap J_{i,j}.
$$
Then
$$
\mu(S_i)=\Pro\left({\dot{\bigcup}}_{k=1}^\Delta V_{i,k}\right)
           =\sum_{k=1}^\Delta \mu(V_{i,k}).
$$
Similarly, $\tilde{S}_i$ is the disjoint union of
 $\tilde{V}_{i,k}= \tilde{J}_{i,k}\setminus \bigcup_{j=1}^{k-1} \tilde{J}_{i,k}\cap\tilde{J}_{i,j}$,
 $k=1,\dots,\Delta$.
Then
$$
F^{-k}V_{i,k}=F^{-k}J_{i,k}\setminus\bigcup_{j=1}^{k-1}F^{-k}\left(J_{i,k}\cap J_{i,j}\right)
=\tilde J_{i,k}\setminus\bigcup_{j=1}^{k-1}\tilde J_{i,k}\cap \tilde J_{i,k-j}
=\tilde{V}_{i,k}.
$$
where we have used that
$F^{-k}J_{i,k}=\tilde J_{i,k}$ and $F^{-k}\left(J_{i,k}\cap J_{i,j}\right)=\tilde J_{i,k}\cap \tilde J_{i,k-j}, 0\le j\le k-1$.
Therefore, by the invariance of the measure $\mu(\tilde{V}_{i,k})=\mu(V_{i,k})$ and consequently
$$
\mu(S_i) =\sum_{k=1}^\Delta \mu(V_{i,k})
= \sum_{k=1}^\Delta \mu(\tilde{V}_{i,k})
= \mu(\tilde{S}_i)\label{symmetry}.
$$
We therefore obtain
$$
b^-_i = \mu\left( \bigcup_{k=1}^{\Delta}T^{-(i-k)}A \cap T^{-i}A\right)
 = \mu\left( \bigcup_{k=1}^{\Delta}T^{-(i+k)}A\cap T^{-i}A \right)
 =\Pro\left( \{U_m^{i,+}>0\} \cap T^{-i}A\right)
 =b^+_i
 $$

Combining~(i) and~(ii) yields
$$
R_1
\le C\sum_{j=r_A}^\Delta\delta_A(j).
$$

\vspace{3mm}

\noindent\textbf{Estimate of $R_3$: } Now we show that short returns are rare.
We proceed similarly to the estimate of $R_1$.
The set inclusions
\begin{eqnarray*}
\{W_m^i=a\}&\subset &\{\tilde{W_m^i}=a\}\cup\{U_m^i>0\}\\
\{\tilde{W_m^i}=a\}& \subset&\{W_m^i=a\}\cup\{U_m^i>0\}
\end{eqnarray*}
let us estimate
$$
\bigg| \Pro\left(\tilde{W_m^i}=a\right)-\Pro\left(W_m^i=a\right) \bigg|
\le \Pro\left(U_m^i>0 \right)
\le 2\Pro\left( \bigcup_{k=1}^{\Delta}\{I_{i+k}=1\} \right)
\le 2\Delta \mu(A).
$$
Hence
$$
R_3\le 2\Delta \mu(A)^2.
$$

\vspace{3mm}

\noindent\textbf{Estimate of $R_2$: } This is the principal term and the speed of mixing now becomes
relevant.
Recall that $ \tilde{W_m^{i}}(x)=W_m^{i, -}(x) + W_m^{i, +}(x)$ and
\begin{eqnarray*}
R_2&=&\bigg|\Pro\left(\{\tilde{W_m^i}=a\}\cap T^{-i}A\right)-\Pro\left(\tilde{W_m^i}=a\right)\mu(A)\bigg|\\
&=&\bigg|\sum_{\substack{\vec{a}=(a^-,a^+)\\ \text{s.t } |\vec{a}|=a}}\Pro\left(\{W_m^{i,\pm}=a^\pm\}
\cap T^{-i}A\right)
-\sum_{\substack{\vec{a}=(a^-,a^+)\\ \text{s.t } |\vec{a}|=a}}\Pro\left(W_m^{i,\pm}=a^\pm\right)
\mu(A)\bigg|.
\end{eqnarray*}
For each $\vec{a}=(a^-,a^+)$ for which $|\vec{a}|=a$ we have 
$$
\bigg|\Pro\left(\{W_m^{i,\pm}=a^\pm\}\cap T^{-i}A\right)-
\Pro\left(W_m^{i,\pm}=a^\pm\right)\mu(A)\bigg|
\le R_{2,1}+R_{2,2}+R_{2,3}
$$
where
\begin{eqnarray*}
R_{2,1}&=& \bigg|\Pro\left(\{W_m^{i,\pm}=a^\pm\}\cap T^{-i}A\right)-
\Pro\left(\{W_m^{i,+}=a^+\}\cap T^{-i}A\right)\Pro\left(W_m^{i,-}=a^-\right)\bigg|\\
R_{2,2}&=&\bigg|\Pro\left(\{W_m^{i,+}=a^+\}\cap T^{-i}A\right)-
\Pro\left(W_m^{i,+}=a^+\right)\mu(A)\bigg|\Pro\left(W_m^{i,-}=a^-\right)\\
R_{2,3}&=&\bigg|\Pro\left(W_m^{i,+}=a^+\right)\Pro\left(W_m^{i,-}=a^-\right)-
\Pro\left(W_m^{i,\pm}=a^\pm\right)\bigg|\mu(A).
\end{eqnarray*}
We now bound the three terms separately:

\noindent \textbf{Bounds for $R_{2,1}$: }
Due to the mixing property
$$
 \bigg|\Pro\left(\{W_m^{i,\pm}=a^\pm\}\cap T^{-i}A\right)
-\Pro\left(\{W_m^{i,+}=a^+\}\cap T^{-i}A\right)\Pro\left(W_m^{i,-}=a^-\right)\bigg|
\le\phi(\Delta)\Pro\left(W_m^{i,-}=a^-\right)
$$
we obtain
\begin{eqnarray*}
 \bigg|\sum_{\substack{\vec{a}=(a^-,a^+)\\ \text{s.t } |\vec{a}|=a}}\Pro\left(\{W_m^{i,\pm}=a^\pm\}
 \cap T^{-i}A\right)
-\sum_{\substack{\vec{a}=(a^-,a^+)\\ \text{s.t } |\vec{a}|=a}}\Pro\left(\{W_m^{i,+}=a^+\} \cap T^{-i}A\right)\Pro\left(W_m^{i,-}=a^-\right)\bigg|\hspace{-6cm}&&\\
&\le&\sum_{\substack{\vec{a}=(a^-,a^+)\\ \text{s.t } |\vec{a}|=a}}
 \phi(\Delta)\Pro\left(W_m^{i,-}=a^-\right)\\
&\le&\phi(\Delta).
\end{eqnarray*}

\noindent \textbf{Bounds for $R_{2,2}$: }
We have 
\begin{eqnarray*}
R_{2,2}&=&\Pro\left(W_m^{i,-}=a^-\right)\bigg|\Pro\left(\{W_m^{i,+}=a^+\}\cap T^{-i}A\right)
-\Pro\left(W_m^{i,+}=a^+\right)\mu(A)\bigg|\\
&\le&\phi(\Delta)\Pro\left(W_m^{i,-}=a^-\right)\mu(A)\\
\end{eqnarray*}
and therefore
\begin{eqnarray*}
 \sum_{\substack{\vec{a}=(a^-,a^+)\\ \text{s.t } |\vec{a}|=a}}\bigg|\Pro\left(\{W_m^{i,+}=a^+\}\cap T^{-i}A\right)\Pro\left(W_m^{i,-}=a^-\right)-
\Pro\left(W_m^{i,+}=a^+\right)\Pro\left(W_m^{i,-}=a^-\right)\mu(A)\bigg|\hspace{-8cm}&&\\
&\le&\sum_{\substack{\vec{a}=(a^-,a^+)\\ \text{s.t } |\vec{a}|=a}}
 \phi(\Delta)\Pro\left(W_m^{i,-}=a^-\right)\mu(A)\\
&\le&\phi(\Delta)\mu(A).
\end{eqnarray*}\\

\noindent \textbf{Bounds for $R_{2,3}$: }
Here we get 
$$
 \sum_{\substack{\vec{a}=(a^-,a^+)\\ \text{s.t } |\vec{a}|=a}} 
\bigg|\Pro\left(W_m^{i,+}=a^+\right)\Pro\left(W_m^{i,-}=a^-\right)
-\Pro\left(W_m^{i,\pm}=a^\pm\right)\bigg|\mu(A)\le\phi(2\Delta)\mu(A).
$$

Combining the estimates for $R_{2,1}, R_{2,2}$ and $R_{2,3}$ we obtain that 
$$
R_2\le R_{2,1}+R_{2,2}+R_{2,3}\le C\phi(\Delta)
$$

\vspace{3mm}

Finally, putting the error terms $R_1$, $R_2$ and $R_3$ together yields
$$
 \bigg|\Pro(\{W_m^i=a\}\cap T^{-i}A)-\Pro(W_m^i=a)\mu(A)\bigg|
  \le C\inf_{\Delta>0}\left(\mu(A)^2\Delta+\mu(A)\sum_{j=r_A}^\Delta\delta_A(j)
  +\phi(\Delta)\right),
$$
for some $C\in\mathbb R^+$ independent of $A$.
\end{proof}

\noindent {\bf Proof of Theorem~\ref{maintheorem}:} By Proposition~\ref{mixingtheorem}
$$
\xi_a\le C\inf_{\Delta>0}\left(\mu(A)^2\Delta+\mu(A)\sum_{j=r_A}^\Delta\delta_A(j)
  +\phi(\Delta)\right)
$$
and therefore
$$
\epsilon_{a,i}\le \mu(A)+\frac{\xi_a}{\nu(A)}
\le C\inf_{\Delta>0}\left(\mu(A)\Delta+\sum_{j=r_A}^\Delta\delta_A(j)
  +\frac{\phi(\Delta)}{\mu(A)}\right).
$$
Let us note that replacing the value $t$ by $t^*=\left[\frac{t}{\mu(A)}\right]\mu(A)$ results in an error of
order ${\cal O}(\mu(A))$.
With the new estimates for the error term $\epsilon_{a,i}$ in hand we can now use
Lemma~\ref{logsum} to obtain (as $\log m={\cal O}(|\log\mu(A)|)$ and $m=[t/\mu(A)]$)
with $E=\{0,1,\dots,k-1\}$:
$$
\left|\Pro\left(\tau_A^k>\frac{t}{\mu(A)}\right)- \sum_{i=0}^{k-1}e^{-t}\frac{t^i}{i!}\right|
\le Ct (t\vee1)\inf_{\Delta>0}\left(\mu(A)\Delta+\sum_{j=r_A}^\Delta\delta_A(j)
  +\frac{\phi(\Delta)}{\mu(A)}\right)\left|\log\mu(A)\right|.
$$
\qed

\vspace{6mm}

\noindent {\bf Proof of Theorem~\ref{mainresult2}:}
\no \textbf{ \textit{(i) Polynomial mixing: }} 
In the polynomial case where $\phi(k)={\cal O}(k^{-\beta})$ with some $\beta>2$ 
 we have  by assumption $\mu(A_{w})=\mathcal{O}(w^{-\beta})$ which implies that
$\delta_A(j)\le \mathcal{O}((\frac{j}2)^{-\beta})+\phi(\frac{j}2)=\mathcal{O}(j^{-\beta})$,
where we used $w=\frac{j}2$. This gives the estimate
$\sum_{j=r_A}^\Delta\delta_A(j)=\mathcal{O}(r_A^{-(\beta-1)})=\mathcal{O}(n^{-(\beta-1)})$
and consequently
$$
\left|\Pro\left(\tau_A^k>\frac{t}{\mu(A)}\right)-\sum_{i=0}^{k-1}e^{-t}\frac{t^i}{i!}\right|
\le Ct (t\vee1) \inf_{\Delta>0}\left(\Delta \mu(A)+\frac{1}{n^{\beta-1}}+\frac{\Delta^{-\beta}}{\mu(A)}\right)
\left|\log\mu(A)\right|.
$$
In order to optimise $\Delta$ put $\Delta=\frac{1}{\mu(A_n)^\omega}$ for some $\omega\in(0,1)$. 
Then we get
$$
 \inf_{\Delta>0}\left(\Delta \mu(A)+\frac{1}{n^{\beta-1}}+\frac{\Delta^{-\beta}}{\mu(A)}\right)
 \le\mu(A)^{1-\omega}+\frac{1}{n^{\beta-1}}+\mu(A)^{\beta\omega-1}.
$$
The best value for  $w\in(0,1)$ is $\omega=\frac{2}{\beta+1}$ and therefore
$$
 \inf_{0<\omega<1}\left(\mu(A)^{1-\omega}+\frac{1}{n^{\beta-1}}+\mu(A)^{\beta\omega-1}\right)
 \le 2\mu(A)^{\frac{\beta-1}{\beta+1}}+\frac{1}{n^{\beta-1}}
 \le \frac{C}{n^{\beta-1}}\quad\forall n\in\mathbb N,
$$
for some constant $C$. Since by assumption $|\log(\mu(A))|\le Kn^\eta$ we obtain
$$
 \inf_{0<\omega<1}\left(\mu(A)^{1-\omega}+\frac{1}{n^{\beta-1}}+\mu(A)^{\beta\omega-1}\right)
 |\log(\mu(A))| \le C \frac{1}{n^{\beta-1-\eta}}
$$
 for some $C>0$. Finally we obtain
 \begin{equation}\label{final-form}
 \left|\Pro\left(\tau_A^k>\frac{t}{\mu(A)}\right)-\sum_{i=0}^{k-1}e^{-t}\frac{t^i}{i!}\right|
\le Ct (t\vee1)\frac{1}{n^{\beta-1-\eta}}.
 \end{equation}

\vspace{3mm}

\no \textbf{ \textit{(ii) Exponential mixing:}} 
In this case  $\phi(k)={\cal O}(\vartheta^k)$ with $\vartheta<1$ which combined with the
 assumption $\mu(A_w(A))=\mathcal{O}(\vartheta^w)$ implies that
 $\delta_A(j)=\tilde\theta^j$ for some $\tilde\theta<1$ 
 (take e.g.\ $w=\min\{n,\frac{j}2\}$.
 Hence $\sum_{j=r_A}^\Delta\delta_A(j)=\mathcal{O}(\tilde\theta^{r_A})=\mathcal{O}(\theta^n)$ 
 for some $\theta<1$. Hence
\begin{equation}\label{preliminary2}
\left|\Pro\left(\tau_A^k>\frac{t}{\mu(A)}\right)-\sum_{i=0}^{k-1}e^{-t}\frac{t^i}{i!}\right|
\le Ct (t\vee1) \inf_{\Delta>0}\left(\Delta \mu(A)+\theta^n+\frac{\theta^\Delta}{\mu(A)}\right)
\left|\log\mu(A)\right|.
\end{equation}
In order to estimate the RHS let us put $\Delta=(1+\epsilon)\frac{|\log\mu(A)|}{|\log\theta|}$ 
for some $\epsilon>0$. Then
$$
  \inf_{\Delta>0}\left(\Delta \mu(A)+\theta^n+\frac{\theta^\Delta}{\mu(A)}\right)
  = (1+\epsilon)\frac{|\log\mu(A)|}{|\log\theta|}\mu(A)+\theta^n+\mu(A)^\epsilon
$$
and therefore
$$
\inf_{\Delta>0}\left(\Delta \mu(A)+\theta^n+\frac{\theta^\Delta}{\mu(A)}\right)
\left|\log\mu(A)\right| 
\le \left((1+\epsilon)\frac{|\log\mu(A)|}{|\log\theta|}\mu(A)+\theta^n+\mu(A)^\epsilon\right)|\log\mu(A)|.
$$
Since for any $\delta\in(0,1)$ $ |\log x|={\cal{O}}\left(\frac{1}{x^\delta}\right)$ as $ x\rightarrow 0^+$
we obtain $ |\log\mu(A)|\le C \frac{1}{\mu(A)^\delta}$
for some constant $C$ independent of $A$.
Hence, as the measure of cylinder sets decay exponentially fast we obtain
$$
  \inf_{\Delta>0}\left(\Delta \mu(A)+\theta^n+\frac{\theta^\Delta}{\mu(A)}\right)\left|\log\mu(A)\right|
  \le Ce^{-\gamma n}
$$
for some $\gamma>0$. Therefore
$$
  \left|\Pro\left(\tau_A^k>\frac{t}{\mu(A)}\right)-\sum_{i=0}^{k-1}e^{-t}\frac{t^i}{i!}\right|
\le Ct (t\vee1)e^{-\gamma n}.
$$
\hfill\qed


\section{Return Times on Markov Towers}

\subsection{Mixing Properties derived on the Markov Tower}

Let $F$ be a differentiable map on a manifold $M$ and $\Omega_0$ a subset of $M$.
As in~\cite{Y2,Y3} we assume that $\Omega_0$ is partitioned into sets $\Omega_{0,i}, i=1,2,\dots$ so that
there is a return time function $R:\Omega_0\rightarrow\mathbb{N}$ which is constant on the partition elements
$\Omega_{0,i}$ and which satisfies that $F^R$ maps $\Omega_{0,i}$ bijectively to the entire set $\Omega_0$.
Let us put $\Omega_{j,i}=\{(x,j): x\in\Omega_{0,i}\}$ for $j=0,1,\dots, R(\Omega_{0,i})-1$. 
The space $\Omega =\bigcup_{i=1}^\infty\bigcup_{j=0}^{R(\Omega_{0,i})-1}\Omega_{j,i}$ is called a {\em Markov tower} for the map $T$. It has the associated partition 
${\cal A}=\{\Omega_{j,i}:\; 0\le j<R(\Omega_{0,i}), i=1,2,\dots\}$ which typically is countably infinite.
 On the tower $\Omega$ we have the map $T$  which for  $x\in\Omega_{0,i}$ is given by
  $T(x,j)=(x,j+1)$ if $j<R(\Omega_{0,i})-1$ and $T(x,R(\Omega_{0,i})-1)=(F^{R(\Omega_{0,i})},0)$.
  
  For points $x,y\in\Omega_0$ one defines the
function $s(x,y)$ as the largest positive $n$ so that $(T^R)^jx$ and  $(T^R)^jy$ for $0\le j<n$
 lie in the same sub-partition elements, that is $(T^R)^jx, (T^R)^jy\in\Omega_{0,i_j}$ 
 for some $i_0,i_1,\dots,n-1$. 
 
 The space of H\"older continuous functions ${\cal C}_\gamma$ consists of all functions $\varphi$ 
 on $\Omega$ for which $|\varphi(x)-\varphi(y)|\le C_\varphi \gamma^{s(x,y)}$. The norm on ${\cal C}_\gamma$
 is $\|\varphi\|_\gamma=|\varphi|_\infty+C_\varphi$, where $C_\varphi$ is smallest possible.
  
 Let $\nu$ be a finite given `reference' measure on $\Omega$ and assume that the Jacobian
 $JT^R$ with respect to the measure $\nu$ is  H\"older continuous in the following sense:
  There exists a $\gamma\in(0,1)$ so that 
 $$
 \left|\frac{JT^Rx}{JT^Ry}-1\right|\le\mbox{\rm const} \gamma^{s(T^Rx,T^Ry)}
 $$
 for all $x,y\in \Omega_{0,i}$, $i=1,2,\dots$.

 If the return time $R$ is integrable with 
 respect to  $m$ then by~\cite{Y3} Theorem~1 there exists a $T$-invariant probability measure
 $\mu$ (SRB measure) on $\Omega$ which is absolutely continuous with respect to $\nu$. Moreover the 
 density function $h=\frac{d\mu}{d\nu}=\lim_{n\rightarrow\infty}\op^n\lambda$ is H\"older continuous,
 where $\lambda$ can be any initial density distribution in ${\cal C}_\gamma$.
 The transfer operator $\op:{\cal C}_\gamma\rightarrow{\cal C}_\gamma$ is defined by
 $\op\varphi(x)=\sum_{x'\in T^{-1}x}\frac{\varphi(x')}{JT(x')}$, $\varphi\in{\cal C}_\gamma$,
 and has the property that $\nu$ is a fixed point of its adjoint, i.e.\ $\op^*\nu=\nu$.
  In \cite{Y3} Theorem~2(II) the $\mathscr{L}^1$-convergence was proven:
\begin{equation}\label{convergence}
\|\op^k\lambda-h\|_{\mathscr{L}^1}\le p(k) \|\lambda\|_\gamma
\end{equation}
where the `decay function' $p(k)={\cal O}(k^{-\beta})$ if the tail decays polynomially with power $\beta$,
that is if $\nu(R> j)\le \const j^{-\beta}$.
If the return times decay exponentially, i.e.\ if $\nu(R>j)\le\const \vartheta^j$ for some $\vartheta\in(0,1)$, then
there is a $\tilde\vartheta\in(0,1)$ so that
$p(k)\le \const\tilde\vartheta^k$.

 Recall that for each $n\in\mathbb{N}$ the elements of the $n$th join 
 ${\cal A}^n=\bigvee_{i=0}^{n-1}T^{-i}{\cal A}$ 
of the partition ${\cal A}=\{\Omega_{i,j}\}$ are called $n$-cylinders. 
For each $n\in\mathbb N$ the $n$-cylinders ${\cal A}^n $ form a new partition of the space, a refinement of the original partition. The
$\sigma$-algebra  $\f$ generated by all $n$-cylinders ${\cal A}^\ell$, for all $\ell\ge1$, is 
the $\sigma$-algebra of the system $(\Omega,\f,\mu)$. 

  We will need the following standard arithmetic lemma to carry estimates for cylinders
over to union of cylinders.

\begin{lemma} \label{ineq2}
Let $a_1,a_2,\dots$ and $b_1,b_2,\dots$ be positive reals. Then
$$
\left|1-\frac{a_1+a_2+\dots}{b_1+b_2+\dots}\right|\le\sup_i\left|1-\frac{a_i}{b_i}\right|.
$$
\end{lemma}

\begin{proof} If we put $\epsilon=\sup_i\left|1-\frac{a_i}{b_i}\right|$ then we have by assumption
 $(1-\epsilon)b_i\le a_i\le(1+\epsilon)b_i$.
Summation over $i$ yields
$$
(1-\epsilon)\sum_ib_i\le\sum_ia_i\le(1+\epsilon)\sum_ib_i
$$
and therefore
$$
(1-\epsilon)\le\frac{\sum_ia_i}{\sum_ib_i}\le(1+\epsilon) 
$$
which implies the statement.
\end{proof}

\begin{lemma}\label{bound1}
There exists a constant $C_6$ so that $\|{\cal{L}}^n\chi_A\|_\gamma\le C_6$ for all 
$A\in\sigma({\cal A}^n)$ and $n$.
\end{lemma}

\begin{proof}
We first show that 
$$
\left|\log \frac{JT^n(x)}{JT^n(y)}\right|\le c_1\gamma^{s(T^{n}x,T^{n}y)},
$$
for all pairs $x,y\in A$,  $A\in {\cal A}^n$ and $\forall n\in \mathbb{N}$.
For $A\in{\cal A}^n$ and $x,y\in A$ we have $x,y\in\Omega_{i,j}$ for some $i<R_j=R(\Omega_{0,j})$. 
Put $n_0=R_j-i$ and then successively  $n_\ell=R_j-i+\sum_{k=1}^{\ell-1} R_{j_k}$,
where the $j_\ell$ are such that $T^{n_\ell}x\in\Omega_{0,j_\ell}$. Clearly 
$T^{n_\ell}x,T^{n_\ell}y\in \Omega_{k,j_\ell}$ for 
$k<R_{j_\ell}$ for all $\ell$ for which $n_\ell\le n$.
Put $L=\max_{n_\ell\le n}\ell$ and we get from the distortion property
\begin{eqnarray*}
\left|\log\frac{JT^n(x)}{JT^n(y)}\right|&\le&\sum_{k=0}^{L-1}
\left|\log\frac{JT^{R_{j_k}}(T^{n_k}(x))}{JT^{R_{j_k}}(T^{n_k}(y))}\right|\\
 &\le&c_2\sum_{k=0}^{L-1} \gamma ^{s(T^{n_k}(x),T^{n_k}(y))}\\
&\le&c_3 \gamma ^{s(T^{n_L}(x),T^{n_L}(y))}\\
 &\le&c_1 \gamma ^{s(T^{n}(x),T^{n}(y))}                       
\end{eqnarray*}
for some $c_1$. 

Now, if $x,y\in\Omega_{i,j}$ for some $i,j$, then let
$A\in{\cal A}^n$ and  $x',y'\in A$ be so that $T^nx'=x$ and $T^ny'=y$
(for $x',y'$ to exist one needs $A\subset\Omega_{i,j}$). Then we obtain
$$
\frac{\op^nh\chi_A(y)}{\op^nh\chi_A(x)}=\frac{h(y')}{h(x')}\frac{JT^n(x')}{JT^n(y')}
$$
which implies by the above estimate and the regularity of the density $h$ that
$$
\left|\log\frac{\op^nh\chi_A(y)}{\op^nh\chi_A(x)}\right|\le\left|\log\frac{JT^n(x')}{JT^n(y')}\right|
+\left|\log\frac{h(y')}{h(x')}\right| \le c_1 \gamma ^{s(x,y)} + c_4 \gamma ^{s(x',y')}\le c_5 \gamma ^{s(x,y)},
$$
for which we can also write 
$$
\left|1-\frac{\op^nh\chi_A(y)}{\op^nh\chi_A(x)}\right|\le c_6\gamma^{s(x,y)} \hspace{5mm}\forall A\in{\cal A}^n.
$$
Now any $A\in\sigma({\cal A}^n)$ is the disjoint union of some $A_j\in{\cal A}^n$.
We now apply Lemma~\ref{ineq2} with the identification $a_j=\op^nh\chi_{A_j}(x), b_j=\op^nh\chi_{A_j}(y)$.
Since $\op^nh\chi_A=\sum_j\op^nh\chi_{A_j}$ we obtain 
$$
\left|1-\frac{\op^nh\chi_A(y)}{\op^nh\chi_A(x)}\right|\le c_6\gamma^{s(x,y)}
 \hspace{5mm}\forall A\in\sigma({\cal A}^n), \forall x,y\mbox{ in some } \Omega_{i,\ell},\forall n.
$$

Let us note that in particular (cf.~\cite{Y3}~Theorem~1(ii) and Sublemma~1) that 
(as $\sum_{A\in{\cal A}^n}\chi_A=1$)
$$
\left|1-\frac{\op^n1(y)}{\op^n1(x)}\right|\le c_1\gamma^{s(x,y)}.
$$

Since $|{\cal{L}}^n1|_\infty \le1$, we now obtain 
$$
\left|\op^nh\chi_A(x)-\op^nh\chi_A(y)\right|
\le|\op^nh\chi_A(y)|\cdot\left|1-\frac{\op^nh\chi_A(x)}{\op^nh\chi_A(y)}\right|
\le C_6\gamma^{s(x,y)}
$$
for some constant $C_6$. Hence $\op^nh\chi_A\in{\cal C}_\gamma$ and, moreover, 
is bounded in the ${\cal C}_\gamma$-norm uniformly in $A\in\sigma({\cal A}^n)$ and $ n\in\mathbb{N}$.
\end{proof}


\noindent We proceed as in the proof of \cite{Y3} Theorem~3 and put $\lambda=\op^nh\chi_A$ which is a strictly
positive function. Then $\eta=\frac\lambda{\mu(A)}$ is a density function as 
$\nu(\lambda)=\nu(\op^nh\chi_A)=\nu(h\chi_A)=\mu(A)$. Moreover $\|\lambda\|_\gamma$ is by 
Lemma~\ref{bound1} bounded by $C_6$ uniformly in $n$ and $A\in\sigma({\cal A}^n)$. 
Denote by $p(k), k=1,2,\dots$, the rate of the decay of 
correlations which is $p(k)={\cal O}(k^{-\beta})$ if the return times tail decays like $k^{-\beta}$ and
$p(k)={\cal O}(\tilde\vartheta^k)$ for some $\tilde\vartheta\in(0,1)$ if the return times tail
decays exponentially. We obtain
\begin{eqnarray}
\mu(A\cap T^{-k-n}B)-\mu(A)\mu(B)&=&
\nu(h \chi_A(\chi_B\circ T^{k+n}))-\nu(h\chi_A)\nu(h\chi_B)\nonumber\\
&=&\mu(A)\nu( \chi_B\op^k\eta)-\nu(h\chi_B)\nonumber\\
&=&\mu(A)\int\chi_B(\op^k\eta-h)\,d\nu\nonumber\\
&=&\int_B(\op^k\lambda-\mu(A)h)\,d\nu.\label{young.mixing1}
\end{eqnarray}
In particular we thus obtain the estimates using the 
$\mathscr{L}^1$-convergence of $\op^k\eta-h$ from~(\ref{convergence}) yields
\begin{eqnarray}
\left|\mu(A\cap T^{-k-n}B)-\mu(A)\mu(B)\right|
&\le&\mu(A)\int\chi_B|\op^k\lambda-h|\,d\nu\nonumber\\
&\le&\left\{\begin{array}{l}\mu(B)\\
\mu(A)c_1\|\eta\|_\gamma p(k)\end{array}\right.\nonumber\\
&\le&\left\{\begin{array}{l}\mu(B)\\
c_2p(k)\end{array}\right.\label{young.mixing2}
\end{eqnarray}
as $\|\eta\|_\gamma=\frac1{\mu(A)}\|\lambda\|_\gamma\le\frac{C_6}{\mu(A)}$. The upper estimate which 
only uses boundedness of $h$ and the pullbacks of the density $\eta$ is useful for 
small $k$ and $\mu(B)$. In particular this shows that the invariant measure on a Young tower
is $\alpha$-mixing but not $\phi$-mixing.

Denote by $\hat{T}=T^R$ the induced map on $\Omega_0$ given by 
$\hat{T}(x)=T^{R(x)}$ for $x\in\Omega_0$ and extended to the entire
tower by  putting $\hat{T}(x)=T^{R(\Omega_{0,i})-j}x$ for $x\in\Omega_{j,i}$.
Similarly we extend $R$ to the entire space $\Omega$ by putting
$R(x)=R(\Omega_{j,i})-j$ for $x\in\Omega_{j,i}$.
To deal with short returns let $A\subset\Omega$ be a set with period $r_A$ and put 
$\mathscr{S}(A)=\bigcup_j A_j$ for  the smallest disjoint union so that $A\subset\mathscr{S}(A)$, 
where $A_j\in\sigma(\mathcal{A}^{\ell_j})$,
and $\ell_j=\sum_{k=0}^{K_j-1}R(\hat{T}^k\tilde{A}_j)$ (for $K_j\ge1$) is such that
 $\ell_j\le\min(n,r_A)$.


\begin{theorem}\label{young.poisson.exponential}
As described above let $T$ be a map on the Markov Tower structure  $\Omega$
with a reference measure $\nu$ and return time function $R$. Let  $\mu$ be the absolutely continuous 
invariant measure. Then for a sequence $A_n\in\sigma({\cal A}^n)$  the following result
 holds true ($\tau_{A_n}^k$ is the $k^{th}$ entry time to $A_n$):

\noindent {\bf (I)} If $\mu(A_n)\ge e^{-Kn}$, $\mu(\mathscr{S}(A_n))\le e^{-Ln}$ for 
some $0<L\le K$ then 
$$
 \left|\Pro\left(\tau_{A_n}^k>\frac{t}{\mu(A_n)}\right)-e^{-t}\sum_{i=0}^{k-1}\frac{t^i}{i!}\right|
\le C_7 (t\vee1)e^{-G n}\quad\forall t>0 \text{ and } \forall n\in\mathbb N,
$$
for all $G<L$ if $p(k)$ is exponential and $G=\frac1{\beta+1}(\beta L-K)$
if $p(k)\sim k^{-\beta}$ is polynomial with $\beta>K/L$.

\noindent {\bf (II)} If $\mu(A_n)\ge n^{-\kappa}$, $\mu(\mathscr{S}(A_n))\le n^{-\lambda}$ for 
some $1<\lambda\le \kappa$ then 
$$
 \left|\Pro\left(\tau_{A_n}^k>\frac{t}{\mu(A_n)}\right)-e^{-t}\sum_{i=0}^{k-1}\frac{t^i}{i!}\right|
\le C_7 (t\vee1)n^{-\gamma}\quad\forall t>0 \text{ and } \forall n\in\mathbb N,
$$
where $\gamma=\lambda-1$ if $p(k)$ is exponential and
 $\gamma=\frac{\beta\lambda -\kappa}{\beta+1}$
if $p(k)\sim k^{-\beta}$ is polynomial of order $\beta>\kappa/\lambda$.
\end{theorem}

\noindent  Note that in both cases, exponentially and polynomially decreasing sets $A_n$
and $\mathscr{S}(A_n)$, the lowest possible bound for the value $\beta$ is $1$ for polynomially
decaying return times tail $\nu(R>n)\sim n^{-\beta}$. In these cases one must
have $K=L$ (exponential case) or $\kappa=\lambda$ (polynomial case).

\subsection{Return times distribution}

Here again we denote by $p(k), k=1,2,\dots$, the rate of the decay of 
correlations as in~(\ref{young.mixing2}), that is $p(k)={\cal O}(k^{-\beta})$ 
if the return times tail decays like $k^{-\beta}$ and
$p(k)={\cal O}(\tilde\vartheta^k)$ for some $\tilde\vartheta\in(0,1)$ if the return times tail
decays exponentially. Let us now prove the main result for Markov towers.

\begin{theorem}\label{young.poisson}
Let  $T\colon\Omega\to\Omega$ be a Markov tower as above 
with a `reference measure' $m$ and a return time function $R$. Let  $\mu$ be the absolutely continuous 
invariant measure for $T$ and $p(k), k=1,2,\dots$, the rate of the decay of correlations.

Let  $A\in\sigma({\cal A}^n)$. Then for all $\Delta$ ($n<\Delta<\!\!<m$) and $m\ge t$:
$$
\left|\Pro(W_m\in E)-\mu_0(E)\right|\le \const\left(\Delta\mu(\mathscr{S}(A))+(2+t)\frac{p(\Delta-n)}{\mu(A)}\log m\right)
$$
\end{theorem}

\begin{proof}
As before we put  $W_m=\sum_{j=1}^m\chi_A\circ T^j$ and 
$W_m^i=\sum_{\substack{1\le j\le m \\ j\neq i}}\chi_A\circ T^j$.
We have to estimate the following quantity:
$$
\left|\Pro(W_m\in E)-\mu_0(E)\right|= \sum_{i=1}^{m}p_i\sum_{a=0}^{m}f(a+1)\epsilon_{a,i},
$$ 
where  
$$
\epsilon_{a,i}=\left|\Pro(W_m=a)-\Pro(W_m=a+1|I_i=1)\right|
\le\left|\Pro(W_m=a)-\Pro(W_m^i=a)\right| +\frac{\xi_a}{\mu(A)}   \label{xi_alpha}, 
$$          
and 
$$
\xi_a=\max_i\left|\Pro(\{W_m^i=a\}\cap T^{-i}A)-\Pro(W_m^i=a) \mu(A)\right|. 
$$
Clearly
$$
\left|\Pro(W_m=a)-\Pro(W_m^i=a)\right|\le \Pro(I_i=1)=\mu(A)
$$
which leaves us to estimate  $\xi_a$ and to execute the sum over $a$ where we will
use the bounds from Lemma~\ref{bound1} for $f$.

Let $\Delta<\!\!<m$ be the halfwith of the `gap' and for $i\in(0,m]$ define as before
\begin{align}W_m^{i, -}&=\sum\limits_{j=1}^{i-(\Delta+1)}\chi_A\circ T^j, &
W_m^{i, +}&=\sum\limits_{j=i+\Delta+1}^{m}\chi_A\circ T^j, \nonumber \\
 U_m^{i, -}&=\sum\limits_{j=i-\Delta}^{i-1}\chi_A\circ T^j, &
U_m^{i,+}&=\sum\limits_{j=i+1}^{i+\Delta}\chi_A\circ T^j, \nonumber \\
 U_m^{i}&= U_m^{i, -} + U_m^{i,+}, &
 \tilde{W_m^{i}}&= W_m^i-U_m^{i} =W_m^{i, -} + W_m^{i, +} \nonumber
 \end{align}
(with the obvious modifications if $i<\Delta$ or $i>m-\Delta$).
For $a\in[0,m]$ we have
\begin{eqnarray*}
\Pro(\{W_m=a+1\}\cap T^{-i}A)&=&\Pro(\{W_m^i=a\}\cap T^{-i}A) \\
&=&\sum_{\substack{\vec{a}=(a^-,a^{0,-},a^{0,+},a^+)\\ \text{s.t } |\vec{a}|=a}}
\Pro\big(\{W_m^{i, \pm}=a^\pm\}\cap \{U_m^{i,\pm}=a^{0,\pm}\}\cap T^{-i}A \big)                                      
\end{eqnarray*}
where the terms inside the sum are measures of intersections of five sets. Then
$$
 \Pro\left(\{W_m^i=a\}\cap T^{-i}A\right)-\Pro\left(W_m^i=a\right)\mu(A)=R_1(a)+R_2(a)+R_3(a),
$$
where
\begin{eqnarray*}
R_1(a)&=&\Pro\left(\{W_m^i=a\}\cap T^{-i}A\right)-\Pro\left(\{\tilde{W_m^i}=a\}\cap T^{-i}A\right)\\
R_2(a)&=&\Pro\left(\{\tilde{W_m^i}=a\}\cap T^{-i}A\right)-\Pro\left(\tilde{W_m^i}=a\right)\Pro\left(I_i=1\right)\\
R_3(a)&=&\left(\Pro\left(\tilde{W_m^i}=a\right)-\Pro\left(W_m^i=a\right)\right)\mu(A)
\end{eqnarray*}
estimated separately as follows in increasing order of difficulty.

\vspace{3mm}

\noindent\textbf{Estimate of $R_3$: } We first show that short returns are rare.
The set inclusions
\begin{eqnarray*}
\{W_m^i=a\}&\subset &\{\tilde{W_m^i}=a\}\cup\{U_m^i>0\}\\
\{\tilde{W_m^i}=a\}& \subset&\{W_m^i=a\}\cup\{U_m^i>0\}
\end{eqnarray*}
let us estimate
$$
\bigg| \Pro\left(\tilde{W_m^i}=a\right)-\Pro\left(W_m^i=a\right) \bigg|
\le \Pro\left(U_m^i>0 \right)
\le 2\Pro\left( \bigcup_{k=1}^{\Delta}\{I_{i+k}=1\} \right)
\le 2\Delta \mu(A).
$$
Hence
$$
|R_3(a)|\le 2\Delta \mu(A)^2
$$
for every $a=0,\dots,m$.

\vspace{3mm}

\noindent\textbf{Estimate of $R_1$: } Here we show that short returns are rare when 
conditioned on $T^{-i}A$. Observe that
\begin{eqnarray*}
\{W_m^i=a\}\cap T^{-i}A
&\subset& \left(\{\tilde{W_m^i}=a\}\cap T^{-i}A\right)\cup\left(\{U_m^i>0\}\cap T^{-i}A\right)\\
\{\tilde{W_m^i}=a\}\cap T^{-i}A&\subset& \left(\{W_m^i=a\}\cap T^{-i}A\right)
\cup\left(\{U_m^i>0\}\cap T^{-i}A\right).
\end{eqnarray*}
Since $U_m^i>0$ implies that either $U_m^{i,+}>0$ or $U_m^{i,-}>0$ we get
$$
\big|\Pro\big(\{W_m^i=a\}\cap T^{-i}A\big)-\Pro\big(\{\tilde{W_m^i}=a\}\cap T^{-i}A\big)\big|
\le \Pro\big( \{U_m^i>0\}\cap T^{-i}A\big)\le b^-_i+b^+_i
$$
where
$$
b^-_i=\Pro\big( \{U_m^{i,-}>0\}\cap T^{-i}A\big)\quad\text{and}\quad
b^+_i=\Pro\big( \{U_m^{i,+}>0\}\cap T^{-i}A\big).
$$
It was shown in Proposition~\ref{mixingtheorem} that $b_i^+=b_i^-$. 

Now let $\mathscr{S}(A)$ be a disjoint union of cylinders ${A}_j\in\sigma({\cal A}^{\ell_j})$, where 
$\ell_j=\sum_{k=0}^{K_j-1}R(\hat{T}^kA_j)$ for some $K_j\ge1$ is so that $\ell_j\le\min(n,r_A)$.
The set $\mathscr{S}(A)$ is chosend so that it contains $A$ and is a disjoint union of $A_j$. This
can be achieved since if there is a non-empty intersection of some ${A}_j$ with some
other cylinder ${A}_k$, then, say, $\ell_j<\ell_k$ which implies that ${A}_k\subset{A}_j$.
It is then sufficient to retain ${A}_j$ and to omit ${A}_k$.
In order to estimate $\mu({A}_j)$ put $\lambda_{{A}_j}=\op^{\ell_j}h\chi_{{A}_j}$.
Then $\lambda_{{A}_j}(x)=\frac{h(y)}{JT^{\ell_j}(y)}$, where $y\in{A}_j$
is such that $T^{\ell_j}y=x$, and $x$ is any point in $\Omega_0$.
Since by \cite{Y3} Sublemma~2
$$
\left|\log\frac{JT^{\ell_j}(y)}{JT^{\ell_j}(y')}\right|\le c_1 \;\;\; \forall \; y,y'\in{A}_j,
$$
for some $c_1$, and as the density $h\in{\cal C}_\gamma$ is positive, we get
$$
\left|\log\frac{\lambda_{{A}_j}(x)}{\lambda_{{A}_j}(x')}\right|\le c_2 \;\;\; \forall \; x,x'\in\Omega_0,
$$ 
and thus $|\lambda_{{A}_j}|_\infty\in[\frac1{c_3},c_3]\frac1{JT^{\ell_j}(y)}\;\;\; \forall \; y\in{A}_j$.
As a consequence 
$\nu({A}_j)$ is similarly comparable to$\frac1{JT^{\ell_j}(y)}\;\;\; \forall \; y\in{A}_j$
as $T^{\ell_j}:{A}_j\rightarrow\Omega_0$ is one-to-one ($c_3>0$) as $\ell_j=R({A}_j)$.
One also has $|\lambda_{{A}_j}|_\infty\le c_4\mu({A}_j)$.
Clearly $\{\tau_A\le\Delta\}\subset\bigcup_{\ell=r_A}^\Delta T^{-\ell}A$ and thus
$$
\mu(A\cap\{\tau_A\le\Delta\})\le\sum_{\ell=r_A}^\Delta\mu(A\cap T^{-\ell}A),
$$
where we can estimate as follows for $\ell\ge\ell_j$
\begin{eqnarray*}
\mu(A\cap T^{-\ell}A)&\le&\sum_j\mu({A}_j\cap T^{-\ell}A)\\
&=&\sum_j\int_{T^{-(\ell-\ell_j)}A}\lambda_{{A}_j}\,d\nu\\
&\le&\sum_j|\lambda_{{A}_j}|_\infty\nu(T^{-(\ell-\ell_j)}A)\\
&\le&c_5\sum_j\mu({A}_j)\mu(A).
\end{eqnarray*}
Since $\mu(\mathscr{S}(A))=\sum_j\mu({A}_j)$ we obtain
$$
b_i^+=\mu_A(\{\tau_A\le\Delta\})\le \sum_{\ell=r_A}^\Delta\frac{\mu(A\cap T^{-\ell}A)}{\mu(A)}
\le c_5\Delta\mu(\mathscr{S}(A))
$$
and thus 
$$
R_1(a)\le b_i^++b_i^-\le2 c_5\Delta\mu(\mathscr{S}(A))
$$
for all $a\in[0,m]$

\vspace{3mm}

\noindent\textbf{Estimate of $R_2$: } Here the decay of correlations play a central role.
For $ \tilde{W_m^{i}}(x)=W_m^{i, -}(x) + W_m^{i, +}(x)$ we obtain as in 
Proposition~\ref{mixingtheorem}
$$
R_2(a)=\sum_{\substack{\vec{a}=(a^-,a^+)\\ \text{s.t } |\vec{a}|=a}}
\left(\Pro\left(\{W_m^{i,\pm}=a^\pm\}
\cap T^{-i}A\right)-\Pro\left(W_m^{i,\pm}=a^\pm\right)\mu(A)\right)
$$
where $a^-+a^+=a$. As before we split the summands into three separate parts
$R_{2,1}, R_{2,2}, R_{2,3}$ which we sum over $a$ and bound separately as follows.

\noindent \textbf{Bounds for $R_{2,1}$: }
The mixing of sets formula~(\ref{young.mixing1}) gives us
\begin{eqnarray*}
R_{2,1}(a^-,a^+)&=&
\mu\left(\{W_m^{i,\pm}=a^\pm\}\cap T^{-i}A\right)
-\mu\left(\{W_m^{i,+}=a^+\}\cap T^{-i}A\right)\mu\left(W_m^{i,-}=a^-\right)\\
&=&\int_{Y_{a^+}}
\left(\op^{\Delta-n}\lambda_{a^-}-h\mu(X_{a^-})\right)\,d\nu,
\end{eqnarray*}
where $\lambda_{a^-}=\op^{i+n}h\chi_{X_{a^-}}$, $X_{a^-}=\{W_m^{i,-}=a^-\}$
and $Y_{a^+}=T^{\Delta-n}(\{W_m^{i,+}=a^+\}\cap T^{-i}A)$.
According to Lemma~\ref{bound1}
$\|\lambda_{a^-}\|_\gamma\le C_6$ for any value of $a^-, i, m$ and $n$.
Thus, summing over $a=0,\dots,m$, we obtain
\begin{eqnarray*}
\left|\sum_{a=0}^m f(a+1)R_{2,1}(a^-,a^+)\right|&\le&\sum_{a^-,a^+}
\left|f(a^-+a^++1)\int_{Y_{a^+}}\left(\op^{\Delta-n}\lambda_{a^-}-h\mu(X_{a^-})\right)\,d\nu\right|\\
&\le&\sum_{a^+=0}^m\sum_{a^-=0}^m|f(a^-+a^++1)|\varepsilon_{a^-,a^+}
\int_{Y_{a^+}}\left(\op^{\Delta-n}\lambda_{a^-}-h\mu(X_{a^-})\right)\,d\nu
\end{eqnarray*}
where $\varepsilon_{a^-,a^+}$ is the sign of the integral
$\int_{Y_{a^+}}\left(\op^{\Delta-n}\lambda_{a^-}-h\mu(X_{a^-})\right)\,dm$.
We now split the sum over $a^-, a^+$ in geometric progression and use the bounds
on $|f|$ from Lemma~\ref{logsum} to obtain
\begin{eqnarray*}
\left|\sum_{a=0}^m f(a+1)R_{2,1}(a^-,a^+)\right|&\le&
\sum_{k=0}^{[\log_22m]}\sum_{a^-,a^+=0}^{[2m2^{-k}]}\frac{2+t}{a^-+a^++1}
\varepsilon_{a^-,a^+}
\int_{Y_{a^+}}\left(\op^{\Delta-n}\lambda_{a^-}-h\mu(X_{a^-})\right)\,d\nu\\
&&+\sum_{a^-,a^+=0}^{[t]}
\varepsilon_{a^-,a^+}
\int_{Y_{a^+}}\left(\op^{\Delta-n}\lambda_{a^-}-h\mu(X_{a^-})\right)\,d\nu.
\end{eqnarray*}
The first (triple) sum is estimated by $I+II$, where $I$ is for the terms with $\varepsilon=+1$
and $II$ contains the terms for which $\varepsilon=-1$. For every $k$ we use the fact that 
$\frac{2+t}{a^-+a^++1}\le \frac{2+t}{m2^{-k}}$ for $a^-+a^+\in[m2^{-k}, m2^{-(k-1)})$. Hence
\begin{eqnarray*}
I&=&\sum_{k=0}^{[\log_22m]}\frac{2+t}{m2^{-k}}\sum_{a^+=0}^{[2m2^{-k}]}
\sum_{\substack{a^-\in[0,2m2^{-k}]\\ \text{s.t } \varepsilon_{a^-,a^+}=1}}
\int_{Y_{a^+}}\left(\op^{\Delta-n}\lambda_{a^-}-h\mu(X_{a^-})\right)\,d\nu\\
&=&\sum_{k=0}^{[\log_22m]}\frac{2+t}{m2^{-k}}\sum_{a^+=0}^{[2m2^{-k}]}
\int_{Y_{a^+}}\left(\op^{\Delta-n}L_{k,a^+,1}-h\mu(\tilde{X}_{a^+,1})\right)\,d\nu
\end{eqnarray*}
(notice that all terms are positive), where 
$$
L_{k,a^+,1}=\sum_{\substack{a^-\in[0,2m2^{-k}]\\ \text{s.t } \varepsilon_{a^-,a^+}=1}}\lambda_{a^-}
=\op^{i+n}\chi_{\tilde{X}_{a^+,1}}
$$
 and 
$\tilde{X}_{a^+,1}=\bigcup_{\substack{a^-\in[0,2m2^{-k}]\\ \text{s.t } \varepsilon_{a^-,a^+}=1}}X_{a-}$ 
is a disjoint union in $\sigma({\cal A}^{i+n})$. Hence by Lemma~\ref{bound1} we have
$\|L_{k,a^+,1}\|_\gamma\le C_6$ for all values of $a^+, i, n$. We thus obtain
\begin{eqnarray*}
I&\le&\sum_{k=0}^{[\log_22m]}\frac{2+t}{m2^{-k}}\sum_{a^+=0}^{[2m2^{-k}]}
\|L_{k,a^+,1}\|_\gamma p(\Delta-n)\\
&\le&C_6 \sum_{k=0}^{[\log_22m]}\frac{2+t}{m2^{-k}} 2m2^{-k} p(\Delta-n)\\
&\le& c_6 (2+t)p(\Delta-n)\log m .
\end{eqnarray*}
Similarly one estimates the second contribution $II$ by putting
$L_{k,a^+,2}=\sum_{\substack{a^-\in[0,2m2^{-k}]\\ \text{s.t } \varepsilon_{a^-,a^+}=-1}}\lambda_{a^-}
=\op^{i+n}\chi_{\tilde{X}_{a^+,2}}$ where 
$\tilde{X}_{a^+,2}$ is the disjoint union 
$\bigcup_{\substack{a^-\in[0,2m2^{-k}]\\ \text{s.t } \varepsilon_{a^-,a^+}=-1}}X_{a-}$. 
We then get as above in estimating the part $I$ (again for every $k$ we estimate 
$|f(a+1)|\le  \frac{2+t}{m2^{-k}}$ for $a^-+a^+\in[m2^{-k}, m2^{-(k-1)})$):
\begin{eqnarray*}
II&=&\sum_{k=0}^{[\log_22m]}\frac{2+t}{m2^{-k}}\sum_{a^+=0}^{[2m2^{-k}]}
\sum_{\substack{a^-\in[0,m2^{-k}]\\ \text{s.t } \varepsilon_{a^-,a^+}=-1}}
-\int_{Y_{a^+}}\left(\op^{\Delta-n}\lambda_{a^-}-h\mu(X_{a^-})\right)\,d\nu\\
&=&\sum_{k=0}^{[\log_22m]}\frac{2+t}{m2^{-k}}\sum_{a^+=0}^{[2m2^{-k}]}
-\int_{Y_{a^+}}\left(\op^{\Delta-n}L_{k,a^+,2}-h\mu(\tilde{X}_{a^+,2})\right)\,d\nu\\
&\le&\sum_{k=0}^{[\log_22m]}\frac{2+t}{m2^{-k}}\sum_{a^+=0}^{[2m2^{-k}]}
\|L_{k,a^+,2}\|_\gamma p(\Delta-n)\\
&\le& c_6 (2+t)p(\Delta-n)\log m
\end{eqnarray*}
as $\|L_{k,a^+,2}\|_\gamma\le C_6$ by Lemma~\ref{bound1}.

In the same way one estimates the second sum above which does not involve
a sum over $k$:
$$
\sum_{a^-,a^+=0}^{[t]}
\varepsilon_{a^-,a^+}
\int_{Y_{a^+}}\left(\op^{\Delta-n}\lambda_{a^-}-h\mu(X_{a^-})\right)\,d\nu
\le C_6t p(\Delta-n).
$$
These estimates combined yield ($c_7\le2c_6+C_6$)
$$
\left|\sum_{a=0}^m f(a+1)R_{2,1}(a^-,a^+)\right|\le c_7 (2+t)p(\Delta-n)\log m.
$$

\noindent \textbf{Bounds for $R_{2,2}$: }
Here we get
\begin{eqnarray*}
R_{2,2}(a^-,a^+)&=&\left(\mu\left(\{W_m^{i,+}=a^+\}\cap T^{-i}A\right)
-\mu\left(W_m^{i,+}=a^+\right)\mu(A)\right)\mu(W_m^{i,-}=a^-)\\
&=&\mu\left(W_m^{i,-}=a^-\right)\int_{T^{\Delta-n}\{W_m^{i,+}=a^+\}}
\left(\op^{\Delta-n}\lambda_*-h\mu(A)\right)\,d\nu
\end{eqnarray*}
where  $\lambda_*=\op^{i+n}h\chi_{T^{-i}A}$
and therefore we get the following estimate which is independent of the value of $a$:
\begin{eqnarray*}
\left|\sum_{a^-+a^+=a}R_{2,2}(a^-,a^+)\right|&\le& 
\sum_{\substack{\vec{a}=(a^-,a^+)\\ \text{s.t } |\vec{a}|=a}}
\mu(W_m^{i,-}=a^-)\left|\mu\left(\{W_m^{i,+}=a^+\}\cap T^{-i}A\right)-
\mu(W_m^{i,+}=a^+)\mu(A)\right|\\
&\le&\sum_{a^+}
\int_{T^{\Delta-n}\{W_m^{i,+}=a^+\}}
\left|\op^{\Delta-n}\lambda_*-h\mu(A)\right)|\,d\nu\\
&\le&\int_{T^{\Delta-n}\bigcup_{a^+}\{W_m^{i,+}=a^+\}}
\left|\op^{\Delta-n}\lambda_*-h\mu(A)\right)|\,d\nu\\
&\le&C_6 p(\Delta-n)
\end{eqnarray*}
again using the fact that  for different $a^+$ the sets $T^{\Delta-n}\{W_m^{i,+}=a^+\}$ are
disjoint in $\sigma(\bigcup_{\ell=i}^\infty{\cal A}^\ell)$.

\noindent \textbf{Bounds for $R_{2,3}$: }
We proceed as in the estimates about  $R_{2,1}$. Put 
$$
R_{2,3}(a^-,a^+)=\mu(A)\left(\mu\left(W_m^{i,+}=a^+\right)\mu\left(W_m^{i,-}=a^-\right)
-\mu\left(W_m^{i,\pm}=a^\pm\right)\right)
$$
and we obtain in the same way that
$$
\left|\sum_{0\le a^-+a^+\le m}f(a^-+a^++1)R_{2,3}(a^-,a^+)\right|\le c_7\mu(A)(2+t)p(2\Delta-n)\log m.
$$

Combining the estimates for $R_{2,1}, R_{2,2}$ and $R_{2,3}$ we obtain that ($c_8\le2c_7+C_6$)
$$
\left|\sum_{a=0}^mf(a+1)R_2(a)\right|\le c_8(2+t)p(\Delta-n)\log m.
$$
On the other hand, using the estimates on $R_1$ and $R_3$ together with the Lemma~\ref{logsum} we get
$$
\left|\sum_{a=0}^mf(a+1)(R_1(a)+R_3(a))\right|
\le\Delta(2\mu(A)+c_5\mu(\mathscr{S}(A)))\sum_{a=0}^m|f(a+1)|
\le\Delta(2\mu(A)+c_5\mu(\mathscr{S}(A)))(t+(2+t)\log\frac{m}t)
$$
Hence 
 \begin{eqnarray*}
\left|\sum_af(a+1)\xi_a\right|&\le& \left|\sum_{a=0}^mf(a+1)(R_1(a)+R_2(a)+R_3(a))\right|\\
&\le&2\Delta\mu(A)(\mu(A)+c_5\mu(\mathscr{S}(A)))\left(t+(2+t)\log\frac{m}t\right)+c_8(2+t)p(\Delta-n)\log m
\end{eqnarray*}
if $m>t$, and therefore
$$
\left|\Pro(W_m\in E)-\mu_0(E)\right|\le
c_9\Delta\mu(\mathscr{S}(A))\left(t+(2+t)\left|\log\mu(A)\right|\right)
+c_8(2+t)\frac{p(\Delta-n)}{\mu(A)}\log m
$$
as $m=[t/\mu(A)]$ for some $c_8,c_9\in\mathbb R^+$ independent of $A$.
\end{proof} 


\noindent {\bf Proof of Theorem~\ref{young.poisson.exponential}.}
 Optimising the error terms requires  the gaps
 $\Delta=(\mu(\mathscr{S}(A_n))\mu(A_n))^{\frac1{1+\beta}}$.
 We now look at different decay rates, namely the two cases when 
(i) $\mu(A_n)$ decays polynomially and (ii) $\mu(A_n)$ decays exponentially.

\vspace{2mm}

\noindent {\bf (i)} If the target set $A_n$ has polynomially decaying measure, 
$\mu(A_n)\sim n^{-\kappa}$ and $\mu(\mathscr{S}(A_n))\sim n^{-\lambda}$,
then if $p(k)={\cal O}(k^{-\beta})$ and  the gaps $\Delta$
are of the order $n^\frac{\kappa+\lambda}{\beta+1}$
 (where $\kappa/\lambda<\beta$ implies that $\Delta<\!\!\!<m=[t/\mu(A_n)]$). 
If $p(k)={\cal O}(\tilde\vartheta^k)$ is exponentially decaying then the best choice for the gaps
is $\Delta\sim n+\log n$. Hence
$$
  \left\{\begin{array}{lcl}
  p(k)={\cal O}(k^{-\beta})&\Rightarrow&
  \left|\Pro(W_m\in E)-\mu_0(E)\right|\le c_1n^{-\frac{\beta\lambda-\kappa}{\beta+1}}\\
  p(k)={\cal O}(\tilde\vartheta^k)& \Rightarrow&
   \left|\Pro(W_m\in E)-\mu_0(E)\right|\le c_1n^{-(\lambda-1)}
  \end{array} \right.
$$
for some $c_1$.

\vspace{2mm}

\noindent {\bf (ii)} In the case when the return set
 $A_n$ has exponentially decaying measure, $\mu(A_n)\le e^{-Kn}$
 (e.g. single $n$-cylinders) and $\mu(\mathscr{S}(A_n))\le e^{-Ln}$ then
 Theorem~\ref{young.poisson}  implies in the polynomial case $p(k)\sim k^{-\beta}$:
$$
\left|\Pro(W_m\in E)-\mu_0(E)\right|\le c_1(t\vee1)\Delta\mu(\mathscr{S}(A_n))
\le c_2e^{-\frac\beta{1+\beta}L+\frac1{\beta+1}K}\le c_2e^{-G n},
$$
where $G=\frac1{\beta+1}(\beta L-K)$ and in the exponential case $p(k)\sim \tilde\vartheta^k$:
$$
\left|\Pro(W_m\in E)-\mu_0(E)\right|\le c_3(t\vee1)\mu(\mathscr{S}(A_n))\log n
\le c_3e^{-G n},
$$
for  any $G <L$.
\hfill\qed


\end{document}